\documentclass[12pt]{amsart}

\usepackage{amsmath,amssymb,amsbsy,amsfonts,amsthm,latexsym,
                        amsopn,amstext,amsxtra,euscript,amscd,mathrsfs,color,bm,cite}
                       
\usepackage{float} 
\usepackage[english]{babel}
\usepackage{mathtools}
\usepackage{todonotes}
\usepackage{url}
\usepackage[colorlinks,linkcolor=blue,anchorcolor=blue,citecolor=blue,backref=page]{hyperref}
\usepackage{mathrsfs}

\usepackage{enumitem}

\def\le{\leqslant}
\def\ge{\geqslant}

\def\leq{\leqslant}
\def\geq{\geqslant}

\usepackage{mathtools}
\usepackage{todonotes}
\usepackage[norefs,nocites]{refcheck}

\newtheorem{theorem}{Theorem}
\newtheorem{lemma}[theorem]{Lemma}

\newtheorem{cor}[theorem]{Corollary}

\numberwithin{equation}{section}
\numberwithin{theorem}{section}
\numberwithin{table}{section}

\numberwithin{figure}{section}

\def\\{\cr}
\def\({\left(}
\def\){\right)}
\def\[{\left[}
\def\]{\right]}
\def\<{\langle}
\def\>{\rangle}
\def\fl#1{\left\lfloor#1\right\rfloor}
\def\rf#1{\left\lceil#1\right\rceil}

\def\N{\mathbb{N}}

\def\cE{\mathcal E}

\def \bc{{\mathbf c}}
\def \bmm{{\mathbf m}}
\def \ba{{\mathbf a}}

\def \N{{\mathbb N}}
\def \Z{{\mathbb Z}}

\def \R{{\mathbb R}}
\def \C{{\mathbb C}}

\def \Esf{\mathsf E}

\def\e{{\mathbf{\,e}}}

\def\R{\mathbb{R}}

\def\Z{\mathbb{Z}}

\newcommand{\commM}[2][]{\todo[#1,color=blue!60]{M: #2}}

\def\mand{\qquad \mbox{and} \qquad}

\def\sfS{{\mathsf S}}

\def\fS{{\mathfrak S}}
\def\fT{{\mathfrak T}}

\def \SQk{\fS_k(\ba; M,N,Q)}
\def \SQ3{\fS_3(\ba; M,N,Q)}

\def \SQf {\widetilde\fS_f(\ba; M,N,Q)}
\def \SQm {\fS(\ba, \bmm; M,N,Q)}
\def \TQm {\fT(\ba, \bmm; M,N,Q)}

\def\eqref#1{(\ref{#1})}

\def\cE{{\mathcal E}}

\def\cL{{\mathcal L}}

\def\cS{{\mathcal S}}

\def\cX{{\mathcal X}}
\def\cY{{\mathcal Y}}

\def\e{{\mathbf{\,e}}}

\def\({\left(}
\def\){\right)}
\def\fl#1{\left\lfloor#1\right\rfloor}
\def\rf#1{\left\lceil#1\right\rceil}

\def\mand{\qquad \mbox{and} \qquad}

\begin{document}

\bibliographystyle{plain}

\title[Additive energy and a large sieve inequality]{Additive energy and a large sieve 
inequality for sparse sequences}
\date{\today}

\author[R. C. Baker] {Roger C.~Baker} 
\address{Department of Mathematics, Brigham Young University, 
Provo, UT 84602, USA} 
\email{baker@math.byu.edu} 

\author[M. Munsch]{Marc Munsch}
\address{DIMA, Universit\`{a} degli Studi di Genova, via Dodecaneso 15, 16146 Genova (GE), Italia}
\email{munsch@dima.unige.it}

\author[I. E. Shparlinski]{Igor E. Shparlinski}  

\address{School of Mathematics and Statistics, 
The University of New South Wales, 
Sydney NSW 2052, 
Australia}
\email{igor.shparlinski@unsw.edu.au}




\begin{abstract}
We  consider the large sieve inequality for sparse sequences of moduli and give a general result depending on the additive energy (both symmetric and asymmetric) of the sequence of moduli. For example, in the case of monomials $f(X) = X^k$ this allows us to 
 improve, in some ranges of the parameters, the previous bounds of S. Baier and L. Zhao 
(2005), K.~Halupczok (2012, 2015, 2018) and M.~Munsch  (2020). We also consider moduli defined by  polynomials $f(X) \in \Z[X]$, Piatetski-Shapiro sequences and general convex sequences. We then apply our results to obtain a version of the Bombieri--Vinogradov theorem with  Piatetski-Shapiro moduli improving the level of distribution of R.~C.~Baker (2014). 
\end{abstract}

\maketitle

\tableofcontents

\section{Introduction}

\subsection{General set-up} 
 The large sieve, which originated in the work of Linnik~\cite{Lin}, has became over the last decades an extremely powerful method in number theory, 
 see~\cite[Chapter~9]{FrIw} and~\cite[Chapter~7]{IwKow}. More recently new variants of the large sieve over sparse sequences of moduli, such as squares, have appeared and found numerous applications in arithmetic problems of different flavours such as the distribution of primes in sparse 
 progressions~\cite{BaZha3,Bak2,Bak3,Bak4}, the existence of shifted primes divisible by a large square~\cite{Mat},  the study of Fermat quotients~\cite{BFKS,SFermat}, 
 elliptic curves~\cite{BPS,SZ} and several others. 

To formulate a general form of  a large sieve inequality, we recall that  a set of real numbers $\{x_k:~k =1, \ldots,K\}$, is 
called {\it $\delta$-spaced modulo $1$\/} if $\langle x_k-x_j \rangle \geq \delta$ for all $1 \le j < k\le K$,  where $\langle x\rangle$ denotes the distance of a real number $x$ to its closest integer. Then by a result of Montgomery and Vaughan~\cite[Theorem~1]{MoVau} we have 
\begin{equation}\label{eq:classic} 
\sum_{k=1}^K \left\vert \sum_{n=M+1}^{M+N} a_n \e(x_k n)\right\vert^2 \leq (\delta^{-1}+N)  \sum_{n=M+1}^{M+N} \vert a_n\vert^2, 
\end{equation}
where $\e(z) = \exp(2\pi i z)$ for $z \in \C$.
see also~\cite[Theorem~9.1]{FrIw} or~\cite[Theorem~7.7]{IwKow}.  

The case when the   $\{x_k:~k =1, \ldots,K\}$, is the set of {\it Farey fractions\/}
of order $Q$, that is, $\{a/q;~ \gcd(a,q)=1, \ 1 \le a < q, \ q \le Q\}$, has always been 
of special interest due to the wealth of arithmetic applications, including the
celebrated {\it Bombieri-Vinogradov type theorem\/}. 

 We now consider this question for the sequence of perfect $k$-powers for an integer $k\ge2$. 
 Similarly to the large sieve modulo squares used in~\cite{BFKS} to study $p$-divisibility 
 of Fermat quotients modulo $p$,  results of this type can be used to study the $p^k$-divisibility, 
and perhaps complement some results of  Cochrane,   De Silva and  Pinner~\cite{CoDeSPi}.
Furthermore, it is quite feasible that it can also embedded in the work of 
Matom{\"a}ki~\cite{Mat} and Merikoski~\cite{Mer} (or in a weaker but more robust approach 
of Baier and Zhao~\cite{BaZha3}). In turn, this is expected to lead to showing the infinitude of primes $p$ such that 
$p-1$ is divisible by a large perfect $k$th power (rather than by a large perfect square as in~\cite{BaZha3,Mat,Mer}). 

More generally, large sieve inequalities with any sparse sequence of moduli $\{m_n\}$, which we also consider here, are expected to lead to results about shifted primes divisible by large divisors coming from the sequence $\{m_n\}$, the most studied case being sequences of polynomial moduli. Indeed, the approach to shifted primes with large square divisors
of Baier and  Zhao~\cite{BaZha3} seems to extend to other sequences without appealing 
to their multiplicative properties (while the method of~\cite{Mat,Mer} is more tuned to 
squares and perhaps other perfect powers). 

Here we obtain such a result for  Piatetski-Shapiro divisors, see Corollary~\ref{cor:PS Div} below. 
This approach has also been successfully used for sequences of multivariate polynomial moduli in a recent work of
Halupczok   and  Munsch~\cite{HaMu}.

Another appearance of such large sieve can be found in a question of Erd\"os and S\'ark{\" o}zy~\cite{ErdSar} about divisibility properties of sumsets. In the case of square-free numbers Konyagin~\cite{Kon} has shown links between such problems and  $L^1$-norms of exponential sums considered by Balog and Ruzsa~\cite{BalRu}. Most certainly these ideas
extend to $k$-free numbers (that is, to integers which are not divisible by $k$th power of a 
prime).

Given a sequence $\ba = \{a_n\}$ of complex numbers and  positive integers $k$, $M$, $N$ and $Q$, we consider the sum 
 $$
 \SQk =  \sum_{q=1}^Q \sum_{\substack{ a=1 \\ \gcd(a,q)=1}}^{q^k} \left\vert \sum_{n=M+1}^{M+N} a_n \e\left(\frac{a}{q^k}n\right)\right\vert^2. 
 $$
 
 Furthermore, given a polynomial $f(T) \in \Z[T]$ with a positive leading coefficient,  
 we consider more general sums with polynomial moduli $f(q)$, which are defined as follows
  $$
 \SQf =  \sum_{q =1}^Q \sum_{\substack{ a=1 \\ \gcd(a,f(q))=1}}^{f(q) } \left\vert \sum_{n=M+1}^{M+N} a_n \e\left(\frac{a}{f(q)}n\right)\right\vert^2 , 
 $$
 where without any loss of generality we always assume 
 that $f(q) \ge 1$ for any integer $q\ge 1$.  
 
 For a general sequence of $\bmm = \{m_j\}$ of integers, we consider the sums
$$\SQm =  \sum_{j =1}^Q \sum_{\substack{ a=1 \\ \gcd(a,m_j)=1}}^{m_j } \left\vert \sum_{n=M+1}^{M+N} a_n \e\left(\frac{a}{m_j}n\right)\right\vert^2 .
 $$

A large sieve inequality is an estimate of the following kind
\begin{equation}\label{eq:largesieve} 
 \SQk = O\( \Delta_k(N,Q)  \|\ba \|^2 \)
\end{equation} 
and similarly for $\SQf$ and $\SQm$, where 
$$
\|\ba \| =  \(\sum_{n=M+1}^{M+N} |a_n|^2\)^{1/2} 
$$
and $\Delta_k(N,Q) $ is some function of the parameters $N$ and $Q$ (which could both depend on $k$)
and the implied constant may depend on $k$.

In the simplest case $k=1$, the bound
 $$
 \SQk \le \(Q^2+N -1\) \|\ba\|^2
 $$
  is classical  and in fact is a special case of 
the following general  version of the large sieve inequality~\eqref{eq:classic}.

Here we are mostly interested in the case $k\ge 2$. For $k=2$ the best known result 
is due to Baier and Zhao~\cite{BaZha2}.

 \subsection{Previous results}\label{previous}
We start with an observation due to Zhao~\cite{Zhaoacta}, that the classical large sieve inequality~\eqref{eq:classic} 
implies~\eqref{eq:largesieve} with 
\begin{equation}\label{eq:trivialdelta} 
\Delta_k(N,Q) = \min\left\{Q^{2k}+N,Q(Q^k+N)\right\}.
\end{equation} 
Zhao~\cite{Zhaoacta} also conjectures that we can take 
\begin{equation}\label{eq:conj} 
\Delta_k(N,Q) =(Q^{k+1}+N) N^{o(1)}
\end{equation} 
 in~\eqref{eq:largesieve}, which is based on the heuristic that  the fractions with power denominator are 
 sufficiently regularly spaced.  Note that this conjecture is nontrivial only for 
\begin{equation}\label{eq:crit range} 
Q^k \leq N \leq Q^{2k}
\end{equation} 
as otherwise it follows from~\eqref{eq:trivialdelta}.
A recent result of Kerr~\cite{KerrB} gives a version of the conjecture~\eqref{eq:conj}  with respect to 
the $L^1$-norm.

Several authors have obtained improvements of~\eqref{eq:trivialdelta} in the critical 
range~\eqref{eq:crit range}, we refer to~\cite{Mun} 
for a short survey and comparison of various bounds. 

 First,  Zhao~\cite[Theorem~3]{Zhaoacta} has presented an inequality of type~\eqref{eq:largesieve} with 
\begin{equation}\label{eq:Zhaobound} 
\Delta_k(N,Q)  = Q^{k+1}+\(NQ^{ 1-1/\kappa_k }+N^{1-1/\kappa_k }Q^{1+k/\kappa_k}\)N^{o(1)}
\end{equation} 
where 
\begin{equation}\label{eq:kappa} 
\kappa_k = 2^{k-1}.
\end{equation} 

 Baier and Zhao~\cite[Theorem~1]{BaZha1} have shown that we can take 
\begin{equation}\label{BaierZhao} 
\Delta_k(N,Q) = \(Q^{k+1}+N+N^{1/2}Q^k\) N^{o(1)}
\end{equation} 
which improves~\eqref{eq:Zhaobound} in the range 
$$Q^{2k-2+2k/\kappa_k} \le N \le Q^{2k}.
$$

 These results have been sharpened in a series of works of Halupczok~\cite{KarinIJNT,Karinquart,Karinsurvey,preKarin}, using the progress made on the Vinogradov mean value theorem  by Bourgain, Demeter and Guth~\cite{Bourgainvino} and
Wooley~\cite{Wooley,Wooley2}. Consequently, we can take
  \begin{equation}\label{Karin2k}\Delta_k(N,Q)=
\(Q^{k+1}+\min\left\{A_k(Q, N), N^{1-\omega_k} Q^{1+(2k-1)\omega_k}\right\}\)  N^{o(1)} \end{equation} 
with 
$$
\omega_k=\frac{1}{(k-1)(k-2)+2}
$$ 
and
$$
A_k(Q,N)=N Q^{1-1/(k(k-1))} +N^{1-1/(k(k-1))}Q^{k/(k-1)} .
$$  
In fact, one can use~\eqref{Karin2k} to bound  $\SQf $ in an analogue of~\eqref{eq:largesieve} 
 for any polynomial $f$ of degree $k$,  see~\cite[Section~6]{preKarin}.
 
Recently, Munsch~\cite{Mun} has further refined this estimate and obtained
\begin{equation}\label{Marc} 
\Delta_k(N,Q)  =   Q^{(k+2)/(k+1) + o(1)}N^{1-1/(k(k+1))}.
\end{equation} 

We also note that in the special case of $k=3$, Baier and Zhao~\cite[Theorem~2]{BaZha1}  have given the following estimate
$$
\SQ3 \le \(Q^{4}+\max\left\{ N^{9/10} Q^{6/5} , NQ^{6/7}\right\}\)  N^{o(1)}. 
$$
 
 Finally we mention that we are not aware of any large sieve estimates with arbitrary sequences, that is, 
 for $\SQm$ which depend on some additive properties  of the sequence of moduli $\bmm$, in particular, 
 on its {\it additive energy} as in this work, see~\eqref{eq:addenergy}  and~\eqref{eq:asymenergy} below.

 \subsection{New results} 
 Let us introduce the following quantities. We define the additive energy of a finite set $\cS\subseteq \mathbb{R}$ to be 
 \begin{equation}\label{eq:addenergy} 
 \Esf^+(\cS)=\#\{ (s_1,t_1,s_2,t_2) \in \cS^4:~s_1+t_1= s_2+t_2\}
 \end{equation}
 and the ``asymmetric'' additive energy with respect to the parameter $h \in \mathbb{Z}$ to be 
 \begin{equation}\label{eq:asymenergy}
 \Esf^+_{h}(\cS)=\#\{ (s_1,t_1,s_2,t_2) \in \cS^4:~s_1+t_1= s_2+t_2+h\}.
  \end{equation}

  It is also convenient to define 
  $$
  \Esf^+_{\star}(\cS)   =  \max_{h\neq 0} \Esf^+_{h}(\cS).
  $$
  In fact it is easy to show that $ \Esf^+_{\star}(\cS)  \le   \Esf^+(\cS)$, however for some sequences 
   $ \Esf^+_{\star}(\cS)$ is much smaller than $ \Esf^+(\cS)$, and in    Theorem~\ref{thm:k>4}  we take 
   advantage of this.

 Now for a sequence $\bmm = \{m_j\}$ of integers and any integer $Q\geq 1$, we denote by $\bmm_Q= \{m_1,\dots,m_Q\}$ 
 the set of its first $Q$ elements.  We now show that a variant of the ideas  of~\cite{CGOS, G-MRST, Kerrboxes}, (rather than using results of~\cite{CGOS, G-MRST, Kerrboxes} directly as in~\cite{Mun}), 
allows us to obtain a general result depending on the additive energies of the truncations of the sequence of moduli.

We also assume that the sequence of moduli $\bmm= \{m_j\}$ satisfies the following additional regularity of growth hypothesis:
there exists   $\alpha > 0$ such that 
 \begin{equation}\label{eq:mj growth}
m_j =  j^{\alpha +o(1)}, \qquad j \to \infty.
\end{equation}

 \begin{theorem}\label{thm:energy} 
  With $\ba= \{a_n\}$, $\bmm=\{m_j\}$, $M$, $N$ and $Q$ as above and also satisfying~\eqref{eq:mj growth} and $Q^{\alpha} \leq N \leq Q^{2\alpha}$, we have 
$$\SQm \le \(N \Esf^+(\bmm_Q)^{1/4}+N^{3/4}Q^{\alpha/2}  \Esf^+_{\star}(\bmm_Q)^{1/4}\) Q^{o(1)}  \|\ba\|^2.$$
\end{theorem} 

Good bounds are known for the additive energy of a large class of sequences. For instance, for any {\it convex\/} sequence  of moduli, that is, a sequence  $\bmm= \{m_j\}$ with 
$$
m_j - m_{j-1} < m_{j+1} - m_j, \qquad j = 2, 3, \ldots, 
$$ 
using   the general bound of Shkredov~\cite[Theorem~1]{Shkredovenergy}, which asserts that 
$$
\Esf_\star^+(\bmm_Q) \le \Esf^+(\bmm_Q) \le  Q^{32/13+o(1)}
$$
(see also~\eqref{eq: Energy-Triv} below), 
we deduce the following result.

\begin{cor}\label{cor:convexcase}
Under the conditions of Theorem~\ref{thm:energy}  and assuming that  $\bmm= \{m_j\}$ is a convex sequence, we have
$$\SQm \le N^{3/4}Q^{\alpha/2+8/13}Q^{o(1)} \|\ba\|^2.$$
\end{cor} 

We observe   that the bound of Corollary~\ref{cor:convexcase} is superior to~\eqref{eq:trivialdelta} (taken with 
$k = 
\alpha$) in the range $Q^{2\alpha-20/13} \leq N \leq Q^{2\alpha-32/39}$.

If more information is available about the sequence  $\bmm= \{m_j\}$
then one can also used stronger bounds from~\cite{BHR}.

Furthermore, it follows immediately from the result of Robert and Sargos~\cite[Theorem~2]{RoSa} 
that for any fixed real  $\alpha \ne 0,1$,  for the {\it Piatetski-Shapiro\/}  sequence $m_j = \fl{j^\alpha}$ we have 
 \begin{equation}\label{eq:RS bound}
\Esf_\star^+(\bmm_Q) \le \Esf^+(\bmm_Q) \le  \(Q^{2}  + Q^{4 -\alpha} \) Q^{o(1)}.
\end{equation}

\begin{cor}\label{cor:n-alpha}
Under the conditions of Theorem~\ref{thm:energy}  and assuming that  $\bmm= \{m_j\}$ with  
$m_j = \fl{j^\alpha}$  for any fixed real  $\alpha \ne 0,1$,  we have
$$\SQm \le N^{3/4}\(Q^{(1+\alpha)/2} + Q^{1+\alpha/4} \) Q^{o(1)} \|\ba\|^2.
$$
\end{cor}

 In Section~\ref{PSmoduli}, we show that Corollary~\ref{cor:n-alpha} combined with the ideas of~\cite{BaZha3,Bak2,Bak3} lead to Bombieri-Vinogradov type theorems for primes in progressions with 
Piatetski-Shapiro moduli (see also~\cite{BakFri} for questions of similar flavor).
 
As a consequence of Theorem~\ref{thm:energy} we also obtain new bounds on $\SQk$ for $k \ge 5$. 
Unfortunately the case of $k=4$ is missing a substantial 
ingredient and so we have to exclude it. In the case of $k=3$  our method works but does not improve previous results, see also Section~\ref{sec:comm}. 
 
 \begin{theorem}\label{thm:k>4}
  With $\{a_n\}$, $M$, $N$ and $Q$ as above and $k \ge 5$,  we have 
$$\SQk \le \left(NQ^{1/2}+N^{3/4}Q^{k/2+1/4+1/(2k^{1/2})}\right) Q^{o(1)}  \|\ba\|^2.$$
\end{theorem}

We  now recall the definition~\eqref{eq:kappa}. Our next result is essentially due to Zhao~\cite[Theorem~3]{Zhaoacta}, see~\eqref{eq:Zhaobound} who presented it only 
for monomials. However,  the approach undoubtedly works for any polynomial. 
However since in~\cite{Zhaoacta} only a brief sketch of the proof of~\eqref{eq:Zhaobound} is given, here we present a complete but slightly shorter proof,  which uses a different technique
and which we hope can find other applications. 
Finally, we formulate this bound in full generality for polynomial moduli (this can also be 
obtained via the method of~\cite{Zhaoacta}). 

\begin{theorem}\label{thm:f} Let $f(T) \in \Z[T]$ be of degree $k \ge 2$. 
With $\{a_n\}$,  $M$, $N$ and $Q$  as above we have 
 \begin{align*}
& \SQf \\
& \qquad \le  \(Q^{k+1}  +    \(N Q^{1-1/\kappa_k }
+    N^{1-1/\kappa_k} Q^{1+k/\kappa_k} \) Q^{o(1)} \) \|\ba\|^2.
 \end{align*}
\end{theorem}   

 \subsection{Comparison with previous results} 
 
 As already mentioned, Theorem~\ref{thm:energy} has no predecessors, hence we only discuss 
  Theorems~\ref{thm:k>4} and~\ref{thm:f}. 
 
 To simplify the exposition, here we assume that all implied constants are absolute, while elsewhere in the paper
 they can depend on $k$.
 
The bound of Theorem~\ref{thm:k>4} improves upon~\eqref{Marc} when 
$$ N\geq Q^{\gamma_k},
$$ 
for some $\gamma_k$ with 
 $\gamma_k = 2k-3 + O\(k^{-1/2}\)$ as $k\to \infty$.

Let us remark that the bound~\eqref{Marc} improves~\eqref{BaierZhao} in the range 
$$ Q^{k} \leq N \leq Q^{\lambda_k}$$ and improves~\eqref{Karin2k} in the range
$$ Q^{k+1+ 2/(k-1)}\leq N \leq Q^{\mu_k}.$$ 
for some $\lambda_k$ and $\mu_k$ with 
$$
\lambda _k =  2k-2+O(k^{-1})\mand \mu_k   = 2k-1+O(k^{-3})
$$
as $k \to \infty$.
Our bound is therefore superior to all previous bounds in the range 
\begin{equation}\label{eq:win} 
 Q^{\sigma_k} \leq N \leq Q^{\tau_k}
\end{equation} 
for some $\sigma_k$ and $\tau_k$ with 
$$
\sigma_k = 2k-3+O(k^{-1/2}) \mand \tau_k =  2k-2+O(k^{-1}),
$$ 
as $k \to \infty$.
In particular, direct calculations show that for $k \ge 7$ we have 
$\sigma_k < \tau_k$ and hence the range~\eqref{eq:win} is not empty
(unfortunately for $k=5, 6$  we have $\sigma_k \ge \tau_k$ and thus
the range~\eqref{eq:win} is void). 

We remark that after tedious but elementary calculations, one can easily get explicit expressions 
for $\lambda_k$, $\mu_k$, $\sigma_k$ and $\tau_k$.

 \subsection{Applications to primes in progressions with  Piatetski-Shapiro moduli}\label{PSmoduli}
 Let us fix some $\alpha > 1$ and for a real $R \ge 1$ we consider the set
\begin{equation}
 \label{eq:Small SaR}
 \cS_\alpha(R) =\{ \fl{j^\alpha}:~j \in \N\} \cap [R, 2R].
\end{equation} 
We further set 
$$
M_\alpha(x;R) = \sum_{q \in  \cS_\alpha(R)} \max_{\gcd(a,q)=1} |E(x,q,a)|
$$
with
$$
E(x,q,a) = \sum_{\substack{n \le x\\n \equiv a \pmod q}} \Lambda(n) - \frac{x}{\varphi(q)} , 
$$
where $\varphi(q)$ is the Euler function and $\Lambda(n)$ is the  {\it von Mangoldt function}:
$$
\Lambda(n) =
\begin{cases}
\log p&\quad\text{if $n$ is a power of the prime $p$,}\\
 0 &\quad\text{otherwise.}
\end{cases}
$$

In the above notation we have the following version of the {\it Bombieri--Vinogradov theorem\/} for 
Piatetski-Shapiro moduli, which we derive combining the ideas and results of~\cite{Bak2,Bak3} 
with Theorem~\ref{thm:energy}.

 \begin{theorem}\label{thm:BV4PS}
  For any fixed  $\alpha$ with $1 < \alpha < 9/4$ and $A > 0$, we have  
$$
M_\alpha(x;R) \le \frac{ \#   \cS_\alpha(R) x}{R \cL^A}, 
$$
where $\cL = \log x$ provided that $R = x^\vartheta$ with some fixed $\vartheta < \varPhi(\alpha)$ and $x$ is large enough,  where 
$$
\varPhi(\alpha) = 
\begin{cases}
3\alpha/(10 \alpha - 4), & \text{for}\ 1 < \alpha < 26/23,\\
 13/28, & \text{for}\ 26/23   \le  \alpha < 2,\\
 13\alpha/(34 \alpha - 12), & \text{for}\ 2\le  \alpha <  23/11,\\
7\alpha/(20 \alpha - 10), & \text{for}\  23/11\le   \alpha < 9/4. \\
\end{cases} 
 $$
\end{theorem}

Let $\mathsf{PS}_\alpha(n)$ be the largest divisor of   $n\in \N$ of the form $\fl{j^\alpha}$, 
$j \in \N$.  Repeating the argument of Baier and Zhao~\cite[Section~8]{BaZha3}, we immediately obtain. 

\begin{cor}\label{cor:PS Div}
Under the conditions of Theorem~\ref{thm:BV4PS}, for any  fixed $\vartheta < \varPhi(\alpha)$ there are infinitely many primes $p$ 
with 
$$
\mathsf{PS}_\alpha(p-1) \ge p^\vartheta.
$$
\end{cor} 
The proof follows the same steps as the proof of~\cite[Theorem~5]{BaZha3} for $p-1$ 
in~\cite[Section~8]{BaZha3} with the only difference that instead of using the asymptotic 
formula 
$$
\sum_{y \le j \le 2y} \frac{1}{\varphi(j^2)} = \frac{3} {\pi^2y} 
+ O\(y^{-2} \log y\),
$$
we use the trivial lower bound
$$
\sum_{y \le j \le 2y} \frac{1}{\fl{j^\alpha}} \ge 
\sum_{y \le j \le 2y} \frac{1}{j^\alpha} \ge 2^{-\alpha} y^{1-\alpha} . 
$$

We note that~\cite[Theorem~1.3]{Bak3} gives the bound of Theorem~\ref{thm:BV4PS} for any $\alpha$, provided that  $R \le x^{9/20- \varepsilon}$ with  any fixed $\varepsilon>0$. Thus the novelty of Theorem~\ref{thm:BV4PS} comes from the
inequality $\varPhi(\alpha) > 9/20$ for  $1 < \alpha < 9/4$, improving the previous level of distribution from~\cite[Theorem~1.3]{Bak3}.

Furthermore, based on the above  we can always assume that 
\begin{equation}\label{eq:large R} 
x^{9/20- \varepsilon}\le R \leq x^{1/2-\varepsilon}
\end{equation} 
for some sufficiently small $\varepsilon>0$.  

\section{Preparations}

\subsection{Notation and conventions}
Throughout the paper, the notation $U = O(V)$, 
$U \ll V$ and $ V\gg U$  are equivalent to $|U|\leqslant c V$ for some positive constant $c$, 
which depends on the degree $k$ and, where applies, on the coefficients of the polynomial $f$
and the real parameters $\alpha$, $\varepsilon$ and $A$ (in the proof of  Theorem~\ref{thm:BV4PS}). 

 We also define $U \asymp V$ as an equivalent $U \ll V \ll U$. 

For any quantity $V> 1$ we write $U = V^{o(1)}$ (as $V \to \infty$) to indicate a function of $V$ which 
satisfies $ V^{-\varepsilon} \le |U| \le V^\varepsilon$ for any $\varepsilon> 0$, provided $V$ is large enough. One additional advantage 
of using $V^{o(1)}$ is that it absorbs $\log V$ and other similar quantities without changing  the whole 
expression.  

We also write $u \sim U$ means $U/2 < u \le U$. 

As we have mentioned, we always assume that $f(q) > 0$ for every positive integer $q$.

 \subsection{Number of solutions to some asymmetric Diophantine equations}

For integers $k,U\geq 1$, we introduce the set of powers 
$$
\cS_{U,k}=\{u^k:~1\leq u\leq U \}.
$$
For any integer $h$, we seek a bound on $\Esf^+_{h}(\cS_{U,k})$ which  improves the essentially trivial estimate 
\begin{equation}
\label{eq: Energy-Triv} 
\Esf^+_{h}(\cS_{U,k})\le  \Esf^+(\cS_{U,k}) \le U^{2+o(1)}.
 \end{equation}

It has been shown in~\cite{CKMS} that a result of Marmon~\cite[Theorem~1.4]{Marm}
implies the following estimate. 

\begin{lemma}
\label{lem:Morm}
For a fixed  $k \ge 2$ and uniformly over $h \ne 0$ we have 
$$
\Esf^+_{h}(\cS_{U,k})\le U^{1+2/k^{1/2}+o(1)} .  
$$
\end{lemma}

Lemma~\ref{lem:Morm} gives a nontrivial bound when $k \ge 5$. Unfortunately we do not have a nontrivial 
bound for $k=4$. However, it is also shown in~\cite{CKMS} that the classical argument of Hooley~\cite{Hool} gives a nontrivial bound for $k=3$. We do not state it precisely because it does not imply a large sieve bound superior to the ones recalled in Section~\ref{previous}. 

 \subsection{Distribution of fractional parts and exponential sums} 
 
The following result is well-known and can be found, for example, in~\cite[Chapter~1, Theorem~1]{Mont2}
(which is a  more precise form of the celebrated Erd\"os--Tur\'{a}n inequality).

\begin{lemma}
\label{lem:ET small int}
Let $\gamma_1, \ldots, \gamma_U$ be a sequence of $U$ points of the unit interval $[0,1]$.
Then for any integer $H\ge 1$, and an interval $[\alpha, \beta] \subseteq [0,1]$,
we have
 \begin{align*}
\# \{u =1, \ldots, U:&~\gamma_u  \in [\alpha, \beta]\} - U(\beta - \alpha)\\
\ll \frac{U}{H} + &\sum_{h=1}^H \(\frac{1}{H} +\min\{\beta - \alpha, 1/h\}\)
\left|\sum_{u=1}^U \e\(h\gamma_u\)\right|.
\end{align*}
\end{lemma}

To use Lemma~\ref{lem:ET small int} we also need an estimate on exponential sums
with polynomials,  which is essentially due to Weyl, see~\cite[Proposition~8.2]{IwKow}.

\begin{lemma}
\label{lem:Weyl}
Let $F(X) \in \R[X]$ be a polynomial of degree $k\ge 2$ with the leading coefficient
$\vartheta \ne 0$.
Then
\begin{align*}
 \sum_{u=1}^U & \e\(F(u)\) \\
&  \ll U^{1-k/2^{k-1}}  \(\sum_{-U < \ell_1 , \ldots,  \ell_{k-1}  < U}
\min\{U, \langle \vartheta k!  \ell_1   \ldots   \ell_{k-1}\rangle^{-1}\}\)^{1/2^{k-1}},
\end{align*} where, as before,  $\langle \xi \rangle = \min\{|\xi - k|~:~k\in \Z\}$ denotes the distance between
a real $\xi$ and the closest integer.

\end{lemma}

\subsection{Distribution in boxes and additive energy}
 Let us take a sequence of integers $\bmm=\{m_j\}$ satisfying~\eqref{eq:mj growth}.
  
 For integers $a$, $m$ with $\gcd(a,m)=1$ and $U, V\ge 1$,  we denote by $T_{a}(m;\bmm,U,V)$  the number of solutions to the 
congruence
\begin{equation}
\label{eq: conggen} 
a m_j  \equiv v \pmod m, \qquad 1\le j \le U, \ |v| \le V.
 \end{equation}
 
 We now relate the congruence~\eqref{eq: conggen} to the additive energy of the sequence $\bmm$.

\begin{lemma}\label{lem:boxesgen} 
For any positive integers $U$ and $V$, uniformly over integers $a$ with  $\gcd(a,m)=1$,  we have 
$$
T_{a}(m;\bmm,U,V)   \ll  \Esf^+(\bmm_U)^{1/4}  +   \(U^{\alpha+o(1)}/m+1\)^{1/4}  V^{1/4}  \Esf^+_\star(\bmm_U)^{1/4}.
$$
 \end{lemma}
 
\begin{proof} Observe that  
$$ 
T_{a}(m;\bmm,U,V)^4\le W, 
$$ 
 where $W$ is the number of solutions to the congruence
 \begin{align*}
 m_{j_1} + m_{j_2} -  m_{j_3} - m_{j_4}  &\equiv a^{-1} w \pmod m, \\
 1\le j_1, j_2, j_3, j_4 &\le U, \quad |w| \le 4V.
 \end{align*}
 Hence there is  a set  $\cY \subseteq  \{0, \ldots, m-1\}$ of cardinality $\# \cY = O(V)$, such that 
 \begin{equation}
\label{eq: T4} 
T_{a}(m;\bmm,U,V)^4  \le  \sum_{y \in \cY} W(y) , 
 \end{equation}
  where $W(y)$ is the number of solutions to the congruence
\begin{equation}
\begin{split} 
\label{eq: cong4}
 m_{j_1} + m_{j_2} -  m_{j_3} - m_{j_4} & \equiv  y \pmod m, \\
 1 \le j_1, j_2, j_3, j_4  \le U,  
 \end{split} 
 \end{equation} 
 with a fixed $y$.

 Clearly~\eqref{eq: cong4} implies that 
 $$
   m_{j_1} + m_{j_2} -  m_{j_3} - m_{j_4} = y  + m z
 $$
 for some integer $z = O\(U^{\alpha+o(1)}/m+1\)$. 
 The contribution from the pairs $(y,z) = (0,0)$ is obviously given by $ \Esf^+(\bmm_U)$.
 For other 
$$
O\(\#\cY \(U^{\alpha+o(1)}/m+1\)\)= O\(V\(U^{\alpha+o(1)}/m+1\)\)
$$ 
admissible pairs we remark that~\eqref{eq: T4}  implies 
 $$
T_{a}(m;\bmm,U,V)^4 \ll \Esf^+(\bmm_U) +  V\(U^{\alpha+o(1)}/m+1\) \Esf^+_{\star}(\bmm_U).
$$
The result now follows. 
 \end{proof} 
 
 We remark that for  Lemma~\ref{lem:boxesgen} only the inequality 
 $m_j \le  j^{\alpha +o(1)}$ matters, however for our main results we need the 
 full power of~\eqref{eq:mj growth}.

\subsection{Polynomial values in small boxes}

For integers $a$, $m$ with $\gcd(a,m)=1$ and $U, V\ge 1$ and a polynomial $f(T) \in \Z[T]$,  we denote by $\widetilde T_{a,f}(m;U,V)$  the number of solutions to the 
congruence
\begin{equation}
\label{eq: poly cong} 
a f(u)  \equiv v \pmod m, \qquad 1\le u \le U, \ |v| \le V.
 \end{equation} 

We also  prove a new estimate on $\widetilde T_{a,f}(m;U,V)$ which uses the ideas 
of the proof of~\cite[Theorem~5]{CCGHSZ}, see also~\cite{KerrMoh} for yet another 
approach in the case of prime moduli $m$. We now recall the definition~\eqref{eq:kappa}.

\begin{lemma}
\label{lem:Polynomial Map} Let $f(T) \in \Z[T]$ be of degree $k \ge 2$. 
For  any positive integers $U$ and $V$, 
uniformly over integers $a$ with  $\gcd(a,m)=1$, 
we have 
$$
\widetilde T_{a,f}(m;U,V)\ll \frac{UV}{m} + U^{1-1/\kappa_k}m^{o(1)}
+  U^{1-k/\kappa_k} V^{1/\kappa_k}m^{o(1)}.
$$
\end{lemma}

\begin{proof}
Let $T = \widetilde T_{a,f}(m;U,V)$.   Clearly we can assume that $1\le U,V < m$, as otherwise the result is trivial 
due to $T \ll \min\{U, V\} \le UV/m$.

We interpret the congruence~\eqref{eq: poly cong} 
as a condition on fractional parts $\{af(u)/m\}$, $u =1, \ldots, U$. 

Applying Lemma~\ref{lem:ET small int} to the sequence
of fractional parts $\{af(u)/m\}$, $u =1, \ldots, U$,
with
$$\alpha = 0, \qquad \beta =  V/m, \qquad H = \rf{m/V},
$$
so that we have
$$
\frac{1}{H} +\min\{\beta - \alpha, 1/h\} \ll \frac{V}{m}, 
$$
for $h =1, \ldots, H$, we derive
\begin{equation}
\label{eq:T Weyl}
T \ll \frac{UV}{m} +\frac{V}{m}\sum_{h=1}^H
\left|\sum_{u=1}^U \e( ahf(u)/m)\right|.
 \end{equation} 
Therefore, by Lemma~\ref{lem:Weyl}, we have
 \begin{align*}
T \ll \frac{UV}{m} &+  \frac{ U^{1-k/\kappa_k}V}{m}   \\
  \times &\sum_{h=1}^H
\(\sum_{-U < \ell_1 , \ldots,  \ell_{k-1}  < U}
\min\left\{U, \left \langle \frac{ab}{m}k! h \ell_1   \ldots
 \ell_{k-1}\right\rangle^{-1}\right\}\)^{1/\kappa_k},
\end{align*}
where $b$ is the leading coefficient of $f$. We remove the common divisors introducing 
$$
a_0 = ab/\gcd(m,b) \mand  m_0 = m/\gcd(m,b)
$$
and rewrite the last bound as 
$$
T \ll \frac{UV}{m}  +  \frac{ U^{1-k/\kappa_k}V}{m}  \sum_{h=1}^H W(h)^{1/\kappa_k},
 $$
 where
$$
W(h) = \sum_{-U < \ell_1 , \ldots,  \ell_{k-1}  < U}
\min\left\{U, \left \langle \frac{a_0}{m_0}k! h \ell_1   \ldots
 \ell_{k-1}\right\rangle^{-1}\right\}.
$$

We observe that since $f$ is a fixed polynomial, we have 
$$
m_0 \asymp m.
$$

First we estimate the contribution  from the terms with 
$$ \ell_1   \ldots \ell_{k-1} \equiv 0 \pmod {m_0}.
$$ 
Clearly the product  $z = \ell_1   \ldots \ell_{k-1}$ can take at most $U^{k-1}/m_0$ values.
If $z=0$ then we have at most $(k-1)(2U)^{k-2}$ possibilities for $\(\ell_1,  \ldots, \ell_{k-1}\)$.
For any other $z\ne 0$,  from the well known bound on the divisor function,  see~\cite[Equation~(1.81)]{IwKow},
we obtain $U^{o(1)}$ possibilities for $\(\ell_1,  \ldots, \ell_{k-1}\)$.
Hence in total we have at most 
$$
(k-1)(2U)^{k-2} + U^{k-1+o(1)}/m_0 \ll U^{k-2} + U^{k-1+o(1)}/m = U^{k-2+o(1)}
$$
choices, and each of them  contributes $U$ to the sum over $\ell_1,  \ldots, \ell_{k-1}$.   

Thus we obtain
$$
T \ll\frac{UV}{m}+   \frac{ U^{1-k/\kappa_k}V}{m} H(U^{k-1+o(1)})^{1/\kappa_k}
+  \frac{VU^{1-k/\kappa_k}}{m} W ,
$$
where
$$
W = \sum_{h=1}^H W_0(h)^{1/\kappa_k}
$$
with 
$$ W_0(h) = 
\sum_{\substack{ 0 < |\ell_1|, \ldots,  |\ell_{k-1}| < U\\ \ell_1   \ldots \ell_{k-1} \not \equiv 0 \pmod {m_0}}}
\min\left\{U, \left \langle \frac{a_0}{m_0}k! h \ell_1   \ldots
 \ell_{k-1}\right\rangle^{-1}\right\}.
$$

Recalling the choice of $H$, we derive
\begin{equation}
 \label{eq:T and W}
T \ll \frac{UV}{m}+     U^{1-1/\kappa_k} m^{o(1)} +  \frac{U^{1-k/\kappa_k}V}{m}  W.
\end{equation}
The H\"{o}lder inequality implies the bound
 \begin{align*}
W^{\kappa_k} & \ll H^{\kappa_k-1}  \sum_{h=1}^H W_0(h)\\
&\ll  H^{\kappa_k-1}  \sum_{h=1}^H \sum_{\substack{ 0 < |\ell_1|, \ldots,  |\ell_{k-1}| < U\\ \ell_1   \ldots \ell_{k-1} \not \equiv 0 \pmod {m_0}}}
\min\left\{U, \left \langle \frac{a_0}{m_0}k! h \ell_1   \ldots
 \ell_{k-1}\right\rangle^{-1}\right\}.
\end{align*}
Collecting together the terms with the same value of
$$
z = k! h \ell_1   \ldots  \ell_{k-1} \not \equiv 0 \pmod {m_0}
$$
and recalling the   bound on the divisor function again, 
we conclude that
$$
W^{\kappa_k} \ll H^{\kappa_k-1} m^{o(1)}
\sum_{\substack{|z| < k! HU^{k-1}\\ z  \not \equiv 0 \pmod {m_0}}}
\min\left\{U, \left \langle \frac{a_0}{m_0}z \right\rangle^{-1}\right\}.
$$
Since the sequence $ \langle a_0z/m_0 \rangle$ is periodic with
period $m_0$ and   $HU^{k-1} \ge HU \ge m_0$, using that $\gcd(a_0,m_0) = 1$ we derive
 \begin{align*}
W^{\kappa_k} & \ll H^{\kappa_k-1} m^{o(1)} \( \frac{HU^{k-1}}{m_0}+1\)
\sum_{z=1}^{m_0-1}  \left \langle \frac{a_0}{m_0}z \right\rangle^{-1}\\
& = H^{\kappa_k-1} m^{o(1)} \( \frac{HU^{k-1}}{m_0}+1\)
\sum_{z=1}^{m_0-1}   \left \langle \frac{z}{m_0} \right\rangle^{-1} \\
& \le H^{\kappa_k} U^{k-1}m^{o(1)} +  H^{\kappa_k-1} m^{1+o(1)}   .
 \end{align*}
Thus, recalling the choice of $H$, we derive
 \begin{align*}
W & \le H  U^{(k-1)/\kappa_k}m^{o(1)} + H^{1-1/\kappa_k} m^{1/\kappa_k+o(1)} \\
& =   U^{(k-1)/\kappa_k} V^{-1}m^{1+o(1)}+ V^{-1+1/\kappa_k}m^{1+o(1)},
 \end{align*}
which after the substitution in~\eqref{eq:T and W} concludes
the proof.
\end{proof}

We remark that Halupczok~\cite{preKarin} gives a different bound on the sum 
on the right hand side of~\eqref{eq:T Weyl}, based on 
the optimal form of the Vinogradov mean value 
theorem~\cite{Bourgainvino,Wooley,Wooley2}.
These bounds are used to derive~\eqref{Karin2k}, which in fact 
applies  to any polynomial $f$ of degree $k$, see~\cite[Section~6]{preKarin}.

 \subsection{Distribution of Farey fractions}  
Let us take a sequence of moduli $\bmm= \{m_j\}$ satisfying~\eqref{eq:mj growth}. We introduce a subset of Farey fractions
$$
\cS(\bmm;Q)=\left\{ \frac{a}{m_j}:~\gcd (a,m_j) =1,\ 1\leq a < m_j,\ Q \leq j \leq 2Q\right\}.
$$
It is easy to remark that two distinct elements of $\cS(\bmm;Q)$ are $1/Q^{2\alpha}$ spaced. Following a classical approach (see for instance~\cite{Zhaoacta}), we measure the spacings between these Farey fractions by the quantity
$$
M(\bmm;N,Q)= \max_{x \in \cS(\bmm;Q)}\# \left\{y \in \cS(\bmm;Q):~\langle x-y\rangle <\frac{1}{2N} \right\}.
$$
As noticed in~\cite{Zhaoacta}, any good estimate on this quantity leads to an inequality of type~\eqref{eq:largesieve}. We prove the following bound.

\begin{lemma}\label{lem:fracgen} For any  integers $N$ and $Q$ with~\eqref{eq:crit range}, we have 
$$
M(\bmm;N,Q) \le \Esf^+(\bmm_Q)^{1/4}  +    N^{-1/4+o(1)} Q^{\alpha/2}  \Esf^+_{\star}(\bmm_Q)^{1/4}.
$$
\end{lemma}

\begin{proof}Let $x=a /m_k$ with $\gcd(a,m_k)=1$. 
We would like to estimate the number of elements 
$y=b/m_j$
with $\gcd(b,m_j)=1$ such that
$$
\left\langle \frac{a}{m_k}-\frac{b}{m_j}\right\rangle=\frac{\vert am_j-bm_k\vert}{m_km_j} <1/(2N). 
$$
 
 We now count the number of pairs $(b,j)$ such that for $z=am_j-bm_k$ we have $|z|\le Z$ 
 for some $Z =    Q^{2\alpha +o(1)}/N$.
 This number  does not exceed the number of pairs $(j,z)$ such that 
 $$
 am_j\equiv z \pmod{m_k}, \quad j \le 2Q, \quad  |z|\le Z.
 $$ 
Applying Lemma~\ref{lem:boxesgen}  with parameters $U=2Q$, $V = Z$, and $m=m_k$ 
(thus by~\eqref{eq:mj growth} we have $U^{\alpha}  \le m^{1+o(1)}$),  we deduce 
the desired result. 
\end{proof}

Now, given a polynomial $f(T) \in \Z[T]$ we denote 
$$
\widetilde \cS_f(Q)=\left\{ \frac{a}{f(q)}:~\gcd\(a,f(q)\) =1,\ 1\leq a < f(q),\ Q \leq q \leq 2Q\right\}.
$$
and 
$$
\widetilde M_f(N,Q)= \max_{x \in \widetilde  \cS_f(Q)}\# \left\{y \in \cS_f(Q):~\langle x-y\rangle <\frac{1}{2N} \right\}.
$$

\begin{lemma}\label{lem:poly}  Let $f(T) \in \Z[T]$ be of degree $k \ge 2$. 
For any  integers $N$ and $Q$ with~\eqref{eq:crit range}, we have
$$
\widetilde  M_f(N,Q) \ll 
Q^{k+1}N^{-1}  +   Q^{1-1/\kappa_k+o(1)}
+  Q^{1+k/\kappa_k+o(1)}  N ^{-1/\kappa_k}. 
$$
\end{lemma}

\begin{proof} We first observe 
$$
f(q) \asymp q^k. 
$$
Then proceeding as in the proof of Lemma~\ref{lem:fracgen} and 
using Lemma~\ref{lem:Polynomial Map} instead of Lemma~\ref{lem:boxesgen} 
we obtain the desired result. 
\end{proof}

\section{Proofs of large sieve bounds}

\subsection{Proof of Theorem~\ref{thm:energy}}

Clearly it is enough to consider only a version of $ \SQm$ with summation of $j$  over a dyadic interval, that is,
$$\TQm =  \sum_{j =Q}^{2Q} \sum_{\substack{ a=1 \\ \gcd(a,m_j)=1}}^{q_j } \left\vert \sum_{n=M+1}^{M+N} a_n \e\left(\frac{a}{m_j}n\right)\right\vert^2 .
 $$

 We proceed similarly as in the proof of~\cite[Theorem~1.2]{Mun} and arrive at 
$$ \TQm \ll M(\bmm;N,Q)N \sum_{n=M+1}^{M+N} \vert a_n\vert^2.$$ The result follows then directly using 
Lemma~\ref{lem:fracgen}.

\subsection{Proof of Theorem~\ref{thm:k>4}}
This is direct from Theorem~\ref{thm:energy},  used together 
with~\eqref{eq: Energy-Triv} and Lemma~\ref{lem:Morm}.

\subsection{Proof of Theorem~\ref{thm:f}}
Again, we proceed similarly as in the proof of~\cite[Theorem~1.2]{Mun} and  arrive at 
$$
\SQf \ll  \widetilde M_k(N,Q) N \sum_{n=M+1}^{M+N} \vert a_n\vert^2.
$$  
We now apply Lemma~\ref{lem:poly} and derive the desired result. 

\section{Some large sieve estimates for Piatetski-Shapiro  sequences}
\label{sec:LS PS}
\subsection{Preliminaries} 
In order to use Theorem~\ref{thm:energy} in the proof of Theorem~\ref{thm:BV4PS}
we need to give a bound on the additive energy $\Esf^+(t,R)$ of the sequence
$$
\cS_{\alpha, t}(R) = \left \{ \fl{j^\alpha}/t:~  \fl{j^\alpha} \sim R, \  t \mid \fl{j^\alpha}\right \}. 
$$
In fact we need it for every integer $t \in [1, R^{1/6}]$.
The reason for the appearance of $t$ in our work is that  a Dirichlet  character modulo $\fl{j^\alpha}$ 
is induced by  a {\it primitive\/}  Dirichlet  character modulo $\fl{j^\alpha}/t$. However the quantity  $\Esf^+(t,R)$
needs to be investigated only for 
\begin{equation}
 \label{eq:Small t}
t \ll R^{1/6}. 
\end{equation}

First we estimate the cardinality of  $ \cS_{\alpha, t}(R)$. 

\begin{lemma}\label{lem:Card StR}  
For $t \le R^{1/6}$ we have 
$$
\# \cS_{\alpha, t}(R)  \ll 
 \frac{R^{1/\alpha}} {t } + R^{1/2}. 
$$
\end{lemma}

\begin{proof} 
We note first that (arguing as in~\cite[Equation~(6.4)]{Bak3})
$$
\# \cS_{\alpha, t}(R) = \# \left \{j \in \N: ~j^\alpha \sim R, \  \left\{j^\alpha/t\right\} < 1/t\right \} + O(1). 
$$
By the Erd\"os--Tur\'{a}n inequality, see  Lemma~\ref{lem:ET small int},  we have 
\begin{equation}
 \label{eq:StR via E-T}
\# \cS_{\alpha, t}(R)  -\frac{(2R)^{1/\alpha} - R^{1/\alpha}}{t} \ll  1 + \frac{R^{1/\alpha}}{t}+     
\sum_{1 \le h \le t} \frac{1}{h}  \left| S_h\right | , 
\end{equation}
where 
$$
S_h = \sum_{j \sim R^{1/\alpha}} \e\(h j^\alpha/t\).
$$

To estimate the exponential sums $S_h$ we apply Van der Corput's inequality~\cite[Theorem~2.2]{GrKol} with 
$$
\lambda = \frac{hR^{1-2/\alpha}} {t}, 
$$
which yields
$$
S_h\ll R^{1/\alpha} \lambda^{1/2} + \lambda^{-1/2}
=  \frac{h^{1/2} R^{1/2}} {t^{1/2} } +  \frac{t^{1/2} R^{1/\alpha-1/2}} {h^{1/2} }.
$$
Hence 
$$
\sum_{1 \le h \le t} \frac{1}{h}  \left| S_h\right | \ll  R^{1/2}+ t^{1/2} R^{1/\alpha-1/2}, 
$$
which after substitution in~\eqref{eq:StR via E-T} implies
$$
\# \cS_{\alpha, t}(R)  - \frac{(2R)^{1/\alpha} - R^{1/\alpha}}{t} \ll  R^{1/2}+ t^{1/2} R^{1/\alpha-1/2} +\frac{R^{1/\alpha}}{t}.
$$
Since for $ t < R^{1/3}$ we have
$$
 t^{1/2} R^{1/\alpha-1/2}  \ll  \frac{R^{1/\alpha}}{t}
 $$
 the result  now follows. 
\end{proof} 

Using the trivial bound 
$$
 \Esf^+(t,R) \le \(\# \cS_{\alpha, t}(R)\)^3, 
 $$
we obtain from Lemma~\ref{lem:Card StR}  
\begin{equation}
 \label{eq:E bound 1}
\Esf^+(t,R) \ll  \frac{R^{3/\alpha}} {t^3} + R^{3/2}.
\end{equation}

 On the other hand, the bound~\eqref{eq:RS bound} of Robert and Sargos~\cite[Theorem~2]{RoSa} 
 implies 
\begin{equation}
 \label{eq:E bound 2}
\Esf^+(t,R) \ll  \begin{cases} 
R^{2/\alpha},& \text{if}\ \alpha> 2,\\
R^{4/\alpha-1},& \text{if}\ 2\ge  \alpha> 1.
\end{cases} 
\end{equation}

\subsection{The case $1< \alpha\le 2$} We may assume that $R> x^{9/20-\varepsilon}$ in view of~\eqref{eq:large R}.  

Denote 
$$
\sfS_t = \fS(\ba,  \cS_{\alpha, t}\(R); M,N,(2R)^{1/\alpha} - R^{1/\alpha}\).
$$
The bounds~\eqref{eq:E bound 1} and~\eqref{eq:E bound 2}, together with Theorem~\ref{thm:energy},  imply, respectively,  the following two estimates on
$\sfS_t$:
\begin{equation}
 \label{eq:St bound 1}
\sfS_t\le  \(\frac{R^{3/4\alpha}} {t^{3/4}} + R^{3/8}\) \(N + \frac{N^{3/4} R^{1/2}}{t^{1/2}}\)R^{o(1)} \|\ba \|^2 
\end{equation}
and
\begin{equation}
 \label{eq:St bound 2}
\sfS_t\le  R^{1/\alpha-1/4} \(N + \frac{N^{3/4} R^{1/2}}{t^{1/2}}\) R^{o(1)} \|\ba \|^2 .
\end{equation}
When $N\leq x^{3/5}$,  we can absorb the term $N$ into $N^{3/4} R^{1/2}/t^{1/2}$. Indeed,
$N^{1/4} < x^{3/20} \ll R^{1/2}/t^{1/2}$ since $R^{1/2}/t^{1/2} >  R^{5/12} > x^{3/16}$. Thus, we get
\begin{equation}
 \label{eq:St bound 1bis}
\sfS_t\le \(\frac{N^{3/4}R^{3/4\alpha+1/2}}{t^{5/4}} + \frac{N^{3/4}R^{7/8}}{t^{1/2}}\)R^{o(1)} \|\ba \|^2 
\end{equation}
and
\begin{equation}
 \label{eq:St bound 2bis}
\sfS_t\le  \frac{R^{1/\alpha+1/4}N^{3/4}}{t^{1/2}} R^{o(1)} \|\ba \|^2 .
\end{equation}
from~\eqref{eq:St bound 1} and~\eqref{eq:St bound 2}, respectively.

 Let $\cX_q$ denote the set of all multiplicative characters modulo $q$ and let $\cX_q^*$ be the set 
 of non-principal characters, see~\cite[Chapter~3]{IwKow} for a background. 

Let $\lambda>0$, we now use the bounds~\eqref{eq:St bound 1bis} and~\eqref{eq:St bound 2bis} to estimate the sums
\begin{equation}\label{defT} T(\bc,\lambda)= 
\sum_{q \in \cS_\alpha(R)} \, 
\sum_{\substack{ \chi \in \cX_q^* \\C(\chi) \sim x^{\lambda}}}
\left| \sum_{n \le N} c(n) \chi(n) \right |^2, 
\end{equation} 
where $c_n= 0$ for $\gcd(n,q)>1$,  the set  $\cS_\alpha(R)$ is defined by~\eqref{eq:Small SaR} and as usual $C(\chi)$ is the conductor of the Dirichlet character $\chi$. Indeed there is a primitive character to modulus $\frac{\fl{j^{\alpha}}}{t}$ such that 
$$\chi(n) = \begin{cases} \widetilde\chi(n) & \text{ if } \gcd(n,\fl{j^{\alpha}})=1, \\
0& \text{ otherwise}.
\end{cases}$$   We use Gallagher's inequality~\cite[Equation~(10), Chapter~27]{Dav}. Discarding a factor $\varphi(R)/R$, for $t\mid \fl{j^{\alpha}}$ we have
\begin{equation}\label{Gallagher}  
 \sum_{\widetilde\chi  \in \cX_{\fl{j^{\alpha}}/t}^*} \left| \sum_{n \le N} c(n)\widetilde\chi (n) \right |^2 \leq 
 \sum_{\substack{ a=1 \\\gcd(a,\fl{j^{\alpha}}/t)=1}}^{\fl{j^{\alpha}}/t}\left| \sum_{n=1}^{N}c_n \e\left(\frac{an}{\fl{j^{\alpha}}/t}\right)\right|^2.
 \end{equation}
For those $t=q/C(\chi)$ counted in $ T(\bc,\lambda)$ we have $Rx^{-\lambda} \ll t \ll Rx^{-\lambda}$. 
Noticing that $c_n\widetilde\chi(n)=c_n \chi(n)$ , we see that~\eqref{Gallagher} yields 
\begin{align*}
T(\bc,\lambda) &\leq  \sum_{t \sim Rx^{-\lambda}} \,
\sum_{\substack{\fl{j^{\alpha}} \sim R \\ t\mid \fl{j^{\alpha}}}} \, \sum_{\widetilde\chi  \in \cX_{\fl{j^{\alpha}}/t}^*}  \left| \sum_{n\leq N}c_n \widetilde\chi(n) \right|^2 \\
& \leq \sum_{t \sim Rx^{-\lambda}}\, \sum_{\substack{\fl{j^{\alpha}} \sim R \\ t\mid \fl{j^{\alpha}}}} \,  \sum_{\substack{ a=1 \\\gcd(a,\fl{j^{\alpha}}/t)=1}}^{\fl{j^{\alpha}}/t}\left| \sum_{n=1}^{N}c_n \e\left(\frac{an}{\fl{j^{\alpha}}/t}\right)\right|^2.
\end{align*} Let us assume the following condition   
\begin{equation}\label{condN} N\leq x^{3/5}. 
\end{equation}
Thus, by~\eqref{eq:St bound 1bis} and~\eqref{eq:St bound 2bis} we obtain the two estimates
\begin{equation}\label{eq: Bound T 1}  
\begin{split} 
T(\bc,\lambda) &\ll  \sum_{t \sim Rx^{-\lambda}} \left(\frac{N^{3/4}R^{3/4\alpha+1/2}}{t^{5/4}}+\frac{N^{3/4}R^{7/8}}{t^{1/2}} \right)  \|\bc \|^2 R^{o(1)}  \\
& \ll   \left(N^{3/4}R^{3/4\alpha+1/4}x^{\lambda/4}+ N^{3/4}R^{11/8}x^{-\lambda/2} \right)   \|\bc \|^2  R^{o(1)}
\end{split}
  \end{equation}  
 and 
\begin{equation}\label{eq: Bound T 2}  
\begin{split} 
T(\bc,\lambda) & \ll R^{o(1)}\sum_{t \sim Rx^{-\lambda}} \frac{N^{3/4}R^{1/\alpha+1/4}}{t^{1/2}}  \|\bc \|^2 . \\
& \ll N^{3/4}R^{1/\alpha+3/4+o(1)}x^{-\lambda/2}  \|\bc \|^2 .
 \end{split} 
\end{equation} 

\subsection{The case $ \alpha>2$} 
Here we have to adjust~\eqref{eq:St bound 2bis}   replacing $1/\alpha+1/4$ by $1/2\alpha +1/2$. Thus
\begin{equation}\label{eq: Bound T 3}  
\begin{split} 
T(\bc,\lambda) &  \ll R^{o(1)}\sum_{t \sim Rx^{-\lambda}} \frac{N^{3/4}R^{1/2\alpha+1/2}}{t^{1/2}} \|\bc \|^2 \\[2mm]
& \ll N^{3/4}R^{1/2\alpha+1+o(1)}x^{-\lambda/2} \|\bc \|^2 .
 \end{split} 
\end{equation}

\section{Proof of Theorem~\ref{thm:BV4PS}}

\subsection{Preliminaries}

In this section we assume $\varepsilon$ is sufficiently small and write $\delta = \varepsilon^2$. We write
 $$
 R = x^\vartheta
 $$
 and assume that~\eqref{eq:large R} holds. That is,
 $$
9/20- \varepsilon\le \vartheta \leq 1/2-\varepsilon.
 $$
 
Let $a^*$ (depending on $q$ and $x$) be chosen so that $\gcd(a^*,q)=1$ and
 $$
\max_{\gcd(a,q)=1}\vert E(x,q,a)\vert = |E(x,q,a^*)|.
 $$

We define an `exceptional subset' $\cE_\alpha(R)$ of $S_\alpha(R)$ below, and show that
for any $A>0$
 \begin{equation}
 \label{eq:except sub1}
\sum_{q\in \cE_\alpha(R)} |E(x,q,a^*)| \ll \frac{x\#S_\alpha(R)}{R\cL^A} 
 \end{equation}
and
 \begin{equation} \label{eq:except sub2}
\sum_{q\in S_\alpha(R)\backslash \cE_\alpha(R)} |E(x,q,a^*)| \ll \frac{x \#S_\alpha(R)}{R\cL^A}.
 \end{equation}

For~\eqref{eq:except sub1} it suffices to show that for any $A>0$
 \begin{equation}\label{eq:Ealph(R)}
\# \cE_\alpha(R) \ll \frac{\# S_\alpha(R)}{\cL^A} . 
 \end{equation}
To see this, suppose that~\eqref{eq:Ealph(R)} holds. By the 
Brun-Titchmarsh theorem (see, for example,~\cite[Theorem~6.6]{IwKow} or~\cite[Theorem~3.9]{MoVau2}), we have 
 $$
 E(x,q,a^*) \ll \frac x{\varphi(q)} \ll \frac xR\, \log\cL \quad (q \in S_\alpha(R)).
 $$
Hence we see from~\eqref{eq:Ealph(R)}  that or any $A > 0$, we have
$$
\sum_{q \in \cE_\alpha(R)} |E(x,q,a^*)|  \ll \frac xR\, \log \cL \#\, \cE_\alpha(R)\ll \frac{x\, \#\, S_\alpha(R)}{R\cL^A}\, .
$$

We define $\cE_\alpha(R)$ by assigning $q \in S_\alpha(R)$ to $\cE_\alpha(R)$ if for some Dirichlet character $x \in \cX_q^*$, the $L$-function $L(s,\chi)$ has a zero $\rho$ with
 $$
 \rho \in \left[1 - \frac \varepsilon{144},1\right) \times 
 [-2R, 2R].
 $$
By~\cite[Lemma~5.2]{Bak3}, with the choice $c_4 = \varepsilon/4$, we have the bound~\eqref{eq:Ealph(R)}.
As a consequence of this definition, we have
 \begin{equation}\label{eq:sumnleN}
\sum_{n\le N} \chi(n) n^{-\frac 12 + it} \ll (|t| + 1) N^{\frac 12} x^{- 3\delta} 
 \end{equation}
for $q \in S_\alpha(R)\backslash \cE_\alpha(R)$ and $N \ge x^{\varepsilon/2}$, for any $\chi \in \cX_q^*$. The implied constant depends only on $\alpha$ and $\varepsilon$. To obtain this we argue as in~\cite[Lemma~5]{Bak2} followed by a partial summation as in~\cite[Lemma~6]{Bak2}.

Before we begin the proof of~\eqref{eq:except sub2}, we assemble some results on mean and large values of Dirichlet polynomials.

 \begin{lemma}\label{lem:letqlex} 
Let $q \le x$. Let $a_n$,  $n\sim N$, be complex numbers and let $G = \sum\limits_{n\sim N} |a_n|^2$. We have
 \begin{equation}\label{eq:sum substacks}
\sum_{\substack{\chi \in \cX_q^*\\
C(\chi)\sim x^\lambda}} \ \Biggl|
\sum_{\substack{n \sim N\\
\gcd(n,q)=1}} a_n \chi(n)\Biggr|^2 \ll x^\delta(N + x^\lambda)G.
 \end{equation}
 \end{lemma}

 \begin{proof}
The left-hand side of~\eqref{eq:sum substacks} is bounded by
 $$
\Sigma =  \sum_{\substack{r\mid q\\
 r\sim x^\lambda}} \ \sum_{\substack{\widetilde \chi \in \cX_r^*\\
 \widetilde \chi \text{ primitive}}} \Biggl|
 \sum_{\substack{n\sim N\\
 \gcd(n,q) = 1}} a_n \widetilde \chi(n)\Biggr|^2.
 $$
This can be bounded as 
 $$
\Sigma  \ll \sum_{r\mid q} (N + x^\lambda)G \ll
 x^\delta (N + x^\lambda) G
 $$ 
by~\cite[Theorem~6.2]{Mont} and the bound on   the divisor function,  see~\cite[Equation~(1.81)]{IwKow}. 
 \end{proof}

 \begin{lemma}\label{lem:letq an} 
Let $q$, $a_n$ and $G$ be as in Lemma~\ref{lem:letqlex}. We have, for $V > 0$,
\begin{equation}
\begin{split} 
\# \Biggl\{ \chi \in \cX_q^*: ~C(\chi) \sim x^\lambda , & \ \Biggl|\sum_{\substack{n\sim N\\
\gcd(n,q)=1}} a_n \chi(n)\Biggr| > V\Biggr\}\label{eq:chi in bold X}\\[2mm]
&\ll x^{2\delta}(GNV^{-2} + x^\lambda G^3NV^{-6}). 
 \end{split}
 \end{equation}
 \end{lemma}
 
 \begin{proof}
The left-hand side of~\eqref{eq:chi in bold X} is bounded by
 $$
\Sigma = \sum_{\substack{r\mid q\\
 r\sim x^\lambda}} \#\Biggl\{\widetilde \chi \in \cX_r^* :~\Biggl|
 \sum_{\substack{N < n \le 2N\\
 \gcd(n,q)=1}} a_n \widetilde \chi(n)\Biggr| > V\Biggr\}.
 $$
This can be bounded as 
 $$
\Sigma \ll \sum_{\substack{r\mid q\\
 r\sim x^\lambda}} (GNV^{-2} + r^{1+\delta} G^3NV^{-6})\ll x^{2\delta}(GNV^{-2} + x^\lambda G^3NV^{-6})
 $$
by~\cite[Theorem~9.18]{IwKow} and the bound on   the divisor function,  see~\cite[Equation~(1.81)]{IwKow}. 
 \end{proof}
 
 \begin{lemma}\label{lem:qlex lge 1}
For $q \le x$, $L \ge 1$, $t \in \R$ we have
$$
\#\Biggl\{\chi \in \cX_q^* :~C(\chi) \sim x^\lambda, \  \Biggl|\sum_{\ell \le L} \chi(\ell)\ell^{-\frac 12 - it}\Biggr| \ge U\Biggr\} \ll x^{\lambda + 6\delta}|s|^{1+\delta} U^{-4}.
$$
 \end{lemma}

 \begin{proof}
If $\chi \in \cX_q^*$ is induced by $\widetilde \chi\in \cX_r^*$, $r \sim x^\lambda$, then
 \begin{align*}
 \sum_{\ell \le L}   \chi(\ell)\ell^{-\frac 12 -it} &= \sum_{\substack{\gcd(\ell,q)=1\\
\ell \le L}} \widetilde \chi(\ell) \ell^{-\frac 12-it}\\[2mm]
&= \sum_{\ell \le L} \Biggl(\sum_{\substack{d\mid q\\
d\mid \ell}} \mu(d)\Biggr) \widetilde \chi(\ell)\ell^{-\frac 12-it}\\[2mm]
&= \sum_{d\mid q} \ \frac{\mu(d)\widetilde \chi(d)}{d^{\frac 12 + it}} \
\sum_{k \le L/d} \widetilde \chi(k) k^{-\frac 12-it}.
 \end{align*}
We now see that if 
$$\Bigl|\sum\limits_{\ell \le L}\chi(\ell)\ell^{-\frac 12-it}\Bigr| \ge U,
$$ 
then 
$$\Bigl|\sum_{k\le L/d} \widetilde \chi(k) k^{-\frac 12 - it}\Bigr| \gg Ux^{-\delta}$$ 
for some $d\mid q$. For a given $d$, the number of possible characters $\widetilde \chi$ is $
O\( |s|^{1+\delta} x^{\lambda +\delta} (Ux^{-\delta})^{-4}\)$
by~\cite[Theorem~10.3]{Mont}. Summing over $d\mid q$, we see that the number of possible characters 
$\widetilde \chi$ is $O\(x^{\lambda + 6\delta}|s|^{1+\delta}U^{-4}\)$. Since $\widetilde \chi$ determines $\chi$, the result follows.  \end{proof}

 \subsection{Application of Vaughan's identity and Heath-Brown's decomposition.}\ 
We begin the proof of~\eqref{eq:except sub2} by using Vaughan's identity; 
see~\cite[Chapter~24]{Dav}. Let $Z = Rx^{\varepsilon/4}$. Then
 $$
 \Lambda(n) = a_1(n) + a_2(n) + a_3(n) + a_4(n)
 $$
with
 \begin{gather*}
 a_1(n) = \begin{cases}
 \Lambda(n) & \text{if } n \le Z\\
 0 & \text{if } n > Z,\end{cases}\\[2mm]
 a_2(n) = -\sum_{\substack{mdr =n\\
 m \le Z,\, d\le Z}} \Lambda(m)\mu(d), \ a_3(n) =
 \sum_{\substack{hd = n\\
 d \le Z}} \mu(d) \log h\\
 \intertext{and}
 a_4(n) = -\sum_{\substack{mr=n\\
 m > Z,\, r > Z}} \Lambda(m) \Biggl(\sum_{\substack{
 d\mid r\\
 d\le Z}} \mu(d)\Biggr).
 \end{gather*}
Let
 $$
 E_i(x,q,a) = \sum_{\substack{
 n\le x\\
 n\equiv a\kern-5pt\pmod q}} a_i(n) - \frac 1{\varphi(q)}\
 \sum_{\substack{
 n \le x\\
 \gcd(n,q)=1}} a_i(n).
 $$
For $q \in S_\alpha(R)$, we have
 \begin{align*}
\sum_{i=1}^4 E_i(x, q, a^*) &= \psi(x;q,a^*) - \frac 1{\varphi(q)} \ \sum_{\substack{n\le x\\
\gcd(n,q)=1}} \Lambda(n)\\[2mm]
&= \psi(x; q,a^*) - \frac x{\varphi(q)} + O\left(\frac{x\cL^{-A}}R\right)
 \end{align*}
by the prime number theorem. Thus to prove~\eqref{eq:except sub2} it suffices to show for $1 \le i \le 4$ that for any $A>0$
 \begin{equation}\label{eq:sum|Ei(x,q,a*)|}
\sum_{q\in S_\alpha(R)\backslash \cE_\alpha(R)} |E_i(x,q,a^*)| \ll \frac{x\, \#S_\alpha(R)}{R\cL^A} .
 \end{equation}

The case $i=1$ of~\eqref{eq:sum|Ei(x,q,a*)|} is obvious from the Brun-Titchmarsh theorem  (see, for example,~\cite[Theorem~6.6]{IwKow} or~\cite[Theorem~3.9]{MoVau2}).

 A partial summation, together with an elementary argument, gives that for any $A>0$
 $$
 E_3(x;q,a) \ll Zx^{\varepsilon/2} \ll
 \frac x{R\cL^A},
 $$
which yields~\eqref{eq:sum|Ei(x,q,a*)|} for $i=3$.

For $i=4$, we refer to~\cite[Section~6]{Bak3} for a proof of~\eqref{eq:sum|Ei(x,q,a*)|}. 

Hence, it remains to prove~\eqref{eq:sum|Ei(x,q,a*)|} for $i=2$.

Let
 $$
 I(x,q,a) = \sum_{\substack{
 m,n\le Z\\
 \gcd(mn, q)=1}} \Lambda(m) \mu(n) \Biggl\{\sum_{\substack{
 \ell \le x/mn\\
 \ell mn \equiv a\kern-5pt\pmod q}} 1 - \frac x{qmn}\Biggr\}.
 $$
By the discussion on~\cite[p.~142]{Bak2}, it suffices for the proof of~\eqref{eq:sum|Ei(x,q,a*)|} for $i=2$ to show that for any $A>0$ we have
 $$
 \sum_{q\in S_\alpha(R)\backslash \cE_\alpha(R)}|I(x,q,a^*)|
  \ll \frac{x\,\# S_\alpha(R)}{R\cL^A}.
 $$ 

To treat $I(x,q,a^*)$, we use Heath-Brown's decomposition~\cite{HB} of $\Lambda(m)$ and the variant, used for example in~\cite[Equation~(2.3)]{Bak1}, for the arithmetic function $\mu$. 
Taking $k=3$ in both cases, we see that
 $$
 \Lambda(m) = \sum_{(I_1,\ldots, I_6)} \
 \sum_{\substack{m_i\in I_i\\
 m_1\ldots m_6 = m}} (\log m_1)\mu(m_4)\mu(m_5)
 \mu(m_6) \quad (1 \le m \le Z)
 $$
and
 $$\mu(n) = \sum_{(J_1,\ldots,J_5)} \
 \sum_{\substack{n_j \in J_j\\
 n_1\ldots n_5 = n}} \mu(n_3)\mu(n_4)\mu(n_5)
 \quad (1 \le n \le Z), 
 $$ 
 where the tuples of intervals $(I_1,\ldots, I_6)$ and $(J_1,\ldots,J_5)$ run through some families 
of  cardinalities $O\((\log Z)^{6}\)$ and $O\((\log Z)^{5}\)$, respectively, with 
$$I_i = (a_i,2a_i], \quad i=1, \ldots, 6, \mand 
J_j = (b_j,2b_j], \quad j =1, \ldots, 5,
$$
such that 
$$\prod_{i=1}^6 a_i < Z, \qquad  \prod_{j=1}^5 b_j < Z$$
and  
$$
2a_i \le Z^{1/3}, \quad  i =4,5,6, \mand 2b_j \le Z^{1/3}, \quad j =3,4, 5.
$$

There are $O(\cL^6)$ tuples $(I_1,\ldots, I_6)$ and $O(\cL^5)$ tuples $(J_1,\ldots J_5)$ in these expressions. Now write $\mu(m) = a(m) + b(m)$ with $a(m) = \max\{\mu(m),0\}$. Define
 $$
 r_0(x,q,a,d) = \sum_{\substack{
 \ell \le x\\
 \ell\equiv a\kern-5pt\pmod q\\
 \ell \equiv 0\kern-5pt\pmod d}} 1 - \frac x{qd}\, .
 $$
 
 We have
 \begin{align*}
 I(x,q,a^*) &= \sum_{\substack{
 m\le Z,\, n\le Z\\
\gcd (mn,q)=1}} \Lambda(m)\mu(n)r_0(x,q,a^*,mn)\\[2mm]
&= \sum_{(I_1,\ldots I_6)} \ \sum_{(J_1,\ldots,J_5)} \\
& \qquad 
\sum_{\substack{
m_i\in I_i,\, n_j\in J_j\\
\gcd(m_in_i,q)=1 \\ 1 \le i \le 6, \, 1 \le j \le 5}}
\prod_{i=4}^6 (a(m_i)+b(m_i))  \prod_{j=3}^5 (a(n_j) + b(n_j))\cdot\\[2mm]
&\qquad \qquad \qquad  \qquad \cdot r_0(x,q,a^*, m_1\ldots m_6 n_1\ldots n_5)  \log m_1. 
 \end{align*}

The last expression splits in an obvious way into $O(\cL^{11})$ sums with an attached $\pm$ sign, of the form
 $$
 \Phi(L_1,\ldots, L_{11};q) = \sum_{\substack{
 \ell_i\sim L_i,\, \gcd(\ell_i, q)=1\\
 1 \le i \le 11}} a_1(\ell_1) \cdots a_{11}(\ell_{11})
 r_0(x,q,a^*,\ell_1 \ldots \ell_{11})
 $$ with nonnegative $a_j(\ell_j)$, such that 
$a_i(\ell_i)$ is identically 1 or identically $\log \ell_i$ if $2L_i > Z^{1/3}$; also
$$
\max\{L_1\ldots L_6, L_7\ldots L_{11}\}\ll Z.
$$
 
 We now summarize our work so far.


 \begin{lemma}\label{lem:let phi be as above}
Let $\Phi$ be as above. If for any $A > 0$, 
 \begin{equation}\label{eq:sum q in S alpha(R)}
\sum_{q \in S_\alpha(R)\backslash \cE_\alpha(R)} |\Phi(L_1,\ldots, L_{11};q)| \ll \frac{x \#S_\alpha(R)}{R\cL^A}.
 \end{equation}
 then~\eqref{eq:except sub2} holds. 
 \end{lemma}

Note that~\eqref{eq:sum q in S alpha(R)} is obvious if
 $$
 L_1\ldots L_{11} \le x^{9/20},
 $$
by the argument used to deal with $E_3$. We assume from now on that
 \begin{equation}\label{eq:L1 to L11 > x(9/20)}
L_1 \ldots L_{11} > x^{9/20}.
 \end{equation}

 \subsection{Riesz's means}\label{sec:RieszsMeans}
In order to prove~\eqref{eq:sum q in S alpha(R)} we work with Riesz's means; the advantage of the logarithmic weighting will become clear. Let us generalize $r_0$ by defining for $k \ge 0$,
 $$
 r_k(x,q,a,d) = \frac 1{k!}\ \sum_{\substack{
 \ell\le x\\
 \ell\equiv a\kern-5pt\pmod q\\
 \ell\equiv 0\kern-5pt\pmod d}} \left(\log \frac x\ell\right)^k
 - \frac x{qd}\, .
 $$
Let $u_d\ge 0$ be given $(D_1 < d \le D)$ where $D_1\asymp D$, $D \le x$ and suppose for some absolute constant $B$ that
 $$
 |u_d| \le \tau(d)^B.
 $$
Suppose further that $1 \le k \le 3$ and we have a bound
$$
\sum_{q\in S_\alpha(R)\backslash \cE_\alpha(R)}\Biggl|
\sum_{D_1 < d \le D} u_d r_k(x,q,a^*,d)\Biggr| \ll \frac{x\#S_\alpha(R)}{R\cL^A}\, .
$$
Then, provided $A$ is sufficiently large,
 $$
 \sum_{q\in S_\alpha(R)\backslash \cE_\alpha(R)}\Biggl|
\sum_{D_1 < d \le D} u_d r_{k-1}(x,q,a^*,d)\Biggr| \ll \frac{x\#S_\alpha(R)}{R\cL^{A/3}}\, .
 $$
See~\cite[p.~154]{Bak2},  for the details of a similar deduction Now we see that it suffices to prove that for any $A>0$ 
 $$
\sum_{q\in S_\alpha(R)\backslash \cE_\alpha(R)} 
|\Phi_4(L_1,\ldots,L_{11}; q)| \ll \frac{x\#S_\alpha(R)}{R\cL^A} 
 $$
where
 \begin{align*}
& \Phi_4(L_1,\ldots, L_{11};q) \\
 & \qquad \qquad = \sum_{\substack{
 \ell_i\sim L_i, \, \gcd(\ell_i, q)=1\\
 1 \le i \le 11}} a_1(\ell_1)\ldots a_{11}(\ell_{11})
 r_4(x, q, a^*, \ell_1\ldots \ell_{11}).
 \end{align*}

We now convert this into a form that requires the counting of Dirichlet characters. We write $r_4$ in the form
 \begin{align*}
r_4(x,q,a^*,d) &= \frac 1{24\varphi(q)}\ \sum_{\chi\in \cX_q} \overline\chi(a^*) \chi(d) \sum_{b \le x/d} \chi(b)\left(\log\frac x{bd}\right)^4 - \frac x{qd}\\
&= \frac 1{24\varphi(q)}\ \sum_{\chi \in \cX_q^*} \overline\chi(a^*)\chi(d) \left(\log \frac x{bd}\right)^4 + O\left(\frac{x^\delta}q\right)
 \end{align*}
for $\gcd(d,q)=1$. We set  
\begin{equation}
\label{eq:ud def}
u_d = \sum_{\substack{ d=\ell_1 \ldots \ell_{11} \\
 \ell_i\sim L_i\, \gcd (\ell_i, q)=1\\
 1 \le i \le 11}} a_1(\ell_1)\ldots a_{11}(\ell_{11})
 \end{equation}   
 for $D_1<d \leq D $ with $D=L_1\ldots L_{11}$ and $D_1=2^{-11}D$.

 Dyadically dissecting the values of the conductor $C(\chi)$, and noting that $D < Z^2 < x^{1-\varepsilon/2}$, it suffices to show that for any $A>0$
\begin{equation}
\begin{split} 
\label{eq:sumqinSalphasumchiinX}
\sum_{q\in S_\alpha(R)\backslash \cE_\alpha(R)} \ \sum_{\substack{\chi \in \cX_q^*\\
C(\chi)\sim x^\lambda}} &\Biggl|\sum_{D_1 < d \le D} u_d\chi(d) \sum_{b \le x/d} \chi(b) \left(\log \frac x{bd}\right)^4\Biggr|\\
&\qquad \qquad \ll \frac{x\# S_\alpha(R)}{\cL^A} 
 \end{split}
 \end{equation}
whenever $0 \le \lambda \le \vartheta$.   

We now use the integral formula
 $$
 \int_{\frac 12 - i\infty}^{\frac 12 + i\infty}
 \frac{y^s}{s^5}\, ds = \begin{cases}
 \frac 1{24}\, (\log y)^4 & \text{if } y \ge 1\\
 0 & \text{if } 0 < y < 1 \end{cases}
 $$
(see~\cite[p.~143]{MoVau2}). This gives
 \begin{align*}
 \frac 1{24} \ \sum_{D_1 < d \le D} & u_d\chi(d) \sum_{b \le x/d} \chi(b) \left(\log \frac x{bd}\right)^4\\[2mm]
 &= \int_{\frac 12 - i\infty}^{\frac 12 + i\infty} x^s \sum_{D_1 < d \le D} u_d \chi(d) d^{-s} B(s,\chi) \, \frac{ds}{s^5}
 \end{align*}
with
 $$
 B(s,\chi) = \sum_{b \le x/D_1} \chi(b) b^{-s}.
 $$
It follows that
\begin{align*}
& \sum_{q\in S_\alpha(R)\backslash \cE_\alpha(R)} \    \sum_{\substack{\chi \in \cX_q^*\\
C(\chi)\sim x^\lambda}}\Biggl|\sum_{D_1 < d \le D} u_d\chi(d) \sum_{b \le \frac xd} \chi(b) \left(\log\, \frac x{bd}\right)^4\Biggr|\\[2mm]
&\qquad \qquad \ll x^{1/2} \int_{\frac 12 - i\infty}^{\frac 12 + i\infty} \sum_{q	\in S_\alpha(R)\backslash \cE_\alpha(R)} \\
& \qquad \qquad \qquad \qquad \quad 
\sum_{\substack{\chi \in \cX_q^*\\
C(\chi) \sim x^\lambda}}\Biggl|\sum_{D_1 < d \le D} u_d\chi(d)d^{-s}\Biggr| |B(s,\chi)| \, \frac{|ds|}{|s|^5}\, .
 \end{align*}

Thus, in order to prove~\eqref{eq:sumqinSalphasumchiinX} it suffices to show that for ${\rm Re}(s) = \frac 12$, for any $A>0$ we have
\begin{equation}
\begin{split} 
\label{eq:same|B(s,chi)|}
\sum_{q\in S_\alpha(R)\backslash \cE_\alpha(R)} \ \sum_{\substack{\chi \in \cX_q^*\\
C(\chi)\sim x^\lambda}} &\Biggl|\sum_{D_1 < d \le D}   u_d\chi(d)d^{-s}\Biggr| |B(s,\chi)|\\
&\qquad \quad \ll \# S_\alpha(R)|s|^3x^{1/2} \cL^{-A}  .
 \end{split}
 \end{equation}  

 \subsection{A key result}
 
 To deal with `small' $\lambda$, we prove a result that is a variant of~\cite[Proposition~1]{Bak2}.
 
 \begin{lemma}\label{lem:letM1dotsM11}
Let $M_1, \ldots, M_{11} \in [1,x]$ and suppose that
 $$
 M = \prod_{i=1}^6 M_i \ll x^{\vartheta + \varepsilon/4},
 \ N = \prod_{i = 7}^{11} M_i \ll x^{\vartheta + \varepsilon/4}.
 $$ Let $a_i(m)$, $m\sim M$, satisfy
 $$
 |a_i(m)| \le \log 2m,  \quad  1 \le i \le 11, \ m \sim M_i.
 $$
We further set
 \begin{align*}
&M_i(s,\chi) = \sum_{m_i\sim M_i} a_i(m) \chi(m) m^{-s} \text{ and}\\[2mm]
&L = \frac x{M_1\ldots M_{11}}\, , \quad B(s,\chi) = \sum_{n\le L} \chi(n) n^{-s}. 
 \end{align*}
Let ${\rm Re}(s) = 1/2$ and
 \begin{equation}\label{eq:lambdalemin}
\lambda \le \min \left\{\frac 9{20} \, , \, \frac 56\, (1-\vartheta)\right\}-\varepsilon.
 \end{equation}
Let $q \in S_\alpha(R)\backslash \cE_\alpha(R)$. Then for any $A >0$
 \begin{equation}\label{eq:sumsubstack|B(schi)|}
\sum_{\substack{\chi \in \cX_q^*\\
C(\chi)\sim x^\lambda}}\Biggl|B(s,\chi) \prod_{i=1}^{11} M_i(s,\chi)\Biggr| \ll |s|^3 x^{1/2} \cL^{-A}.
 \end{equation}
 \end{lemma}
 
 \begin{proof}
Let
 $$
 M(s,\chi) =\prod_{i=1}^6  M_i(s,\chi), \quad
 N(s,\chi) =  \prod_{i = 7}^{11}  M_i(s,\chi).
 $$
We have
 $$
 M(s,\chi) \ll M^{1/2} \cL^{11}\, , \, 
 N(s,\chi) \ll N^{1/2} \cL^{11}\, , \,
 B(s,\chi) \ll L^{1/2}.
 $$
Thus the characters $\chi \in \cX_q^*$ for which one of these three Dirichlet polynomials has absolute value less than $x^{-1}$ can be neglected. We partition the remaining characters with $O(\cL^3)$ sets $A_q(U,V,W)$ defined by the inequalities.
 $$
 U < |B(s,\chi)| \le 2U \, , \, V < |M(s,\chi)|
 < 2V \, , \, W < |N(s,\chi)| \le 2W
 $$
where $U \ll L^{1/2}$, $V \ll M^{1/2}\cL^{11}$, $W \ll N^{1/2} \cL^{11}$. To prove~\eqref{eq:sumsubstack|B(schi)|} it suffices to show that for any $A>0$
$$
UVW |A_q(U,V,W)| \ll |s|^3 x^{1/2}\cL^{-A} . 
$$
From Lemmas~\ref{lem:letqlex}, \ref{lem:letq an}  and~\ref{lem:qlex lge 1} applied to $B(s,\chi)$, $M(s,\chi)$, $N(s,\chi)$, $B(s,\chi)^2$ we obtain
 $$
 |A_q(U,V,W)| \ll |s|^2 Px^\delta,
 $$
where
 \begin{align*}
P = \min\Biggl\{\frac{M+x^\lambda}{V^2}\, , \,  \frac{N+x^\lambda}{W^2} \, , \, & \frac{x^\lambda}{U^4}\, ,
 \, \frac M{V^2} + \frac{x^\lambda M}{V^6}\, ,\\
 &\frac N{W^2} + \frac{x^\lambda N}{W^6} \, , \, \frac{L^2}{U^4} + 
 \frac{x^\lambda L^2}{U^{12}}\Biggr\}.
 \end{align*}
Thus it suffices to show that
 $$
 UVWP \ll |s|x^{1/2-2\delta}.
 $$
We consider four cases depending on the size of the parameters.
 \bigskip

\noindent\textbf{Case~1.} $P \le 2V^{-2}M$, $P \le 2W^{-2}N$. We apply~\eqref{eq:sumnleN}; we have $L \gg x^{\varepsilon/2}$. Since $q \in S_\alpha(R)\backslash \cE_\alpha(R)$, we have
 $$
 U \ll |s|\, L^{1/2} x^{-3\delta}
 $$
and
 \begin{align*}
UVWP &\le 2UVW \min\{V^{-2}M, W^{-2}N\}\\
&\le 2U(MN)^{1/2} \ll |s| x^{1/2} \cL^{11} x^{-3\delta} \ll |s| x^{1/2-2\delta}.
 \end{align*}
 \bigskip

\noindent\textbf{Case~2.} $P > 2V^{-2}M$, $P > 2W^{-2}N$. We proceed as in Case~2 
of~\cite[p.~145]{Bak2} with $Q$ replaced by $x^\lambda$. We obtain
$$
P \ll (UVW)^{-1}(x^{\frac 1{16} + \frac{31\lambda}{32}} + x^{1/20 +\lambda})\ll (UVW)^{-1} x^{\frac 12 - 2\delta} 
$$
since $\lambda \le \frac 9{20} - \varepsilon$.
 \bigskip

\noindent\textbf{Case~3.} $P > 2V^{-2}M$, $P \le 2W^{-2}N$.

We proceed as in Case~3 of~\cite[p.~145]{Bak2}, again with $Q$ replaced by $x^\lambda$. We obtain
$$
P  \ll (UVW)^{-1}(x^{\frac 18 + 7\lambda/16} N^{3/8} + x^{1/12+\lambda/2} N^{5/12})\ll (UVW)^{-1} x^{1/2 - 2\delta} 
$$
since $N \le x^{\vartheta + \varepsilon/4}$ and $\lambda \le \frac 56\, (1-\vartheta) - \varepsilon$.
 \bigskip 

\noindent\textbf{Case~4.} $P > 2W^{-2}N$, $P \le 2V^{-2}M$. We proceed as in Case~3, interchanging the roles of $M$ and $N$.

This completes the proof.  \end{proof}  

Summing over $q \in S_\alpha(R)\backslash \cE_\alpha(R)$, we see that Lemma~\ref{lem:letM1dotsM11} implies the bound~\eqref{eq:same|B(s,chi)|} 
whenever~\eqref{eq:lambdalemin} holds. Hence Lemma~\ref{lem:let phi be as above} holds in this case and~\eqref{eq:except sub2} follows.  

It remains to show that~\eqref{eq:same|B(s,chi)|} holds whenever
 \begin{equation}\label{eq:min910,56,1-theta}
\min\left\{\frac 9{20},\, \frac 56\, (1 - \vartheta)\right\} - \varepsilon \le \lambda \le \vartheta.
 \end{equation} Before turning to the proof of this case, let us make the following remark. In order to apply the large sieve  bounds~\eqref{eq: Bound T 1}, \eqref{eq: Bound T 2} and~\eqref{eq: Bound T 3},  we need~\eqref{eq:Small t}. Note that if $t\geq R^{1/6}$ then $x^\lambda \leq R^{5/6}$, or $\lambda \leq \frac{5\vartheta}{6}$. Recalling that $\vartheta \leq 1/2-\varepsilon$, we deduce that 
 $$ \frac{5\vartheta}{6} \leq  \min\left\{\frac 9{20},\, \frac 56\, (1 - \vartheta)\right\} - \varepsilon.$$ Hence the result for these $\lambda$'s follow from  Lemma~\ref{lem:letM1dotsM11} and we can assume in the following that $t\leq R^{1/6}$.

 \subsection{Conclusion of the proof of Theorem~\ref{thm:BV4PS}}
 
 \begin{lemma}\label{lem:K(s chi) = sum n sim K}
Let
 $$
 K(s,\chi) = \sum_{n\sim K} a_n \chi(n) n^{-s} \mand
  H(s,\chi) = \sum_{m\sim H} b_m \chi(m)m^{-s}
 $$
with $a_n = x^{o(1)}$, $b_n = x^{o(1)}$, $a_n = b_n = 0$ for $(n,q) > 1$, $H \le K \le x^{3/5}$, $HK \ll x$. Let ${\rm Re}(s) = 1/2$ and define
 $$
 U(H, K, \lambda) = \sum_{q\in S_\alpha(R)} \
 \sum_{\substack{\chi \in \cX_q^*\\
 C(\chi) \sim x^\lambda}} |H(s,\chi) K(s,\chi)|.
 $$
Then for $1 < \alpha \le 2$ we have
 \begin{equation}
 \label{eq:U(HKlam)llx38+o(1)}
U(H,K, \lambda) \ll x^{3/8 + o(1)} (R^{\frac 3{4\alpha} + \frac 14} x^{\lambda/4} + R^{11/8} x^{-\lambda/2})
\end{equation}
and 
 \begin{equation}
\label{eq:U(HKlam)llx38-lam/2}
U(H,K,\lambda) \ll x^{3/8 - \lambda/2 + o(1)} R^{1/\alpha + 3/4}. 
\end{equation}
For $\alpha > 2$, we have
 \begin{equation}\label{eq:U(HKlam)llx38-lam/2+0)(1)}
U(H,K,\lambda) \ll x^{3/8-\lambda/2 + o(1)} R^{\frac 1{2\alpha} + 1}.
 \end{equation}
 \end{lemma}
 
 \begin{proof}
The three upper bounds~\eqref{eq:U(HKlam)llx38+o(1)}, \eqref{eq:U(HKlam)llx38-lam/2} and~\eqref{eq:U(HKlam)llx38-lam/2+0)(1)} are obtained by applying the Cauchy  inequality in combination with~\eqref{eq: Bound T 1}, \eqref{eq: Bound T 2} and~\eqref{eq: Bound T 3} respectively (note that $H \le K \le x^{3/5}$ so the condition~\eqref{condN} is fulfilled).
 \end{proof}

We now use Lemma~\ref{lem:K(s chi) = sum n sim K} to obtain~\eqref{eq:same|B(s,chi)|} when~\eqref{eq:min910,56,1-theta} holds. In fact, in this part of the argument 
we do not need to discard $\cE_\alpha(R)$ which has only been used in the Case~$1$ of the proof of Lemma~\ref{lem:letM1dotsM11}.

It suffices to prove~\eqref{eq:same|B(s,chi)|} with $B(s,\chi)$ replaced by 
$$
B_1(s,\chi) = \sum\limits_{\ell \sim L_1} \chi(\ell)\ell^{-s},
$$ 
with $1 \le L_1 \le L$.  Using the shape of the coefficients $u_d$ given by~\eqref{eq:ud def}, we factorize 
$$B_1(s,\chi) \sum\limits_{D_1 < d \le D} u_d\chi(d)d^{-s} =N_1(s,\chi) \ldots N_{12}(s,\chi)
$$ where  
 $$ 
 N_i(s,\chi) = \sum_{n\sim N_i} c_{i,n} \chi(n)n^{-s} \, ,
 \qquad N_1 \ge \cdots \ge N_{12}.
 $$

Remark that these coefficients $c_{i,n}$ are identically $1$, or identically $\log n$, if $N_i > Z^{1/3}$. Write $N_i = x^{\beta_i}$. We recall that we may suppose
 $$
 \beta_1 + \cdots + \beta_{12} \ge x^{9/20}
 $$
(see~\eqref{eq:L1 to L11 > x(9/20)}).

Suppose first that $\beta_1 + \beta_2 > 3/5$. We write 
$$N_0(s,\chi) = N_3(s,\chi) \ldots N_{12}(s,\chi)
$$ 
and define
 \begin{align*}
A(U_0, U_1, U_2) &= \{\chi \in \cX_q^*:~q \in S_\alpha(R), C(\chi) \sim x^\lambda,\\[2mm]
&\qquad \qquad U_j < |N_j(s,\chi)| \le 2U_j, \ j=0, 1, 2\}.
 \end{align*} Arguing as in the proof of Lemma~\ref{lem:letM1dotsM11}, it suffices to show that for any $A>0$, 
 \begin{equation}\label{eq:U_0U_1U_2}
U_0 U_1 U_2 \# A(U_0, U_1, U_2) \ll R^{1/\alpha} |s|^2 x^{1/2} \cL^{-A} .
 \end{equation}

Since $N_1 \ge x^{3/10} > Z^{1/3}$, we have
 $$
 \# A(U_0, U_1, U_2) \ll R^{1/\alpha}
 |s|^{1+\delta} x^{\vartheta + \delta} U_1^{-4}
 $$  
from Lemma~\ref{lem:qlex lge 1} (and, if needed, a partial summation). Next,
 $$
 \#\, A(U_0, U_1, U_2) \ll R^{1/\alpha}|s|^{1+\delta}
 x^{\vartheta+\delta} U_2^{-4}
 $$
from Lemma~\ref{lem:qlex lge 1} (if $N_2 > Z^{1/3}$) and Lemma~\ref{lem:letqlex} applied to $N_2(s, \chi)^2$, if $N_2 \le Z^{1/3}$. We have also
 $$
 \#\, A(U_0, U_1, U_2) \ll R^{1/\alpha}
 x^{\vartheta+\delta} U_0^{-2}
 $$
from Lemma~\ref{lem:letqlex}, since $N_0 \ll x^{2/5} \ll x^\vartheta$. Hence
 $$
 \#\, A(U_0, U_1, U_2) \ll R^{1/\alpha}
 |s|^{1+\delta} x^{\vartheta+\delta} (U_1^{-4})^{1/4}
 (U_2^{-4})^{1/4} (U_0^{-2})^{1/2},
 $$
and~\eqref{eq:U_0U_1U_2} follows at once.

Now suppose that $\beta_1 + \beta_2 \le 3/5$. For some integer $k$, $2 \le k \le 12$, we have, by an elementary argument,
 $$
 x^{2/5} \ll \prod_{j=1}^k N_j \ll x^{3/5}
 $$
(compare with~\cite[Lemma~14]{Bak2}). 
We now apply Lemma~\ref{lem:K(s chi) = sum n sim K} with
 $$
 H(s,\chi) = \prod_{j\le k} N_j(s,\chi)\mand
 K(s,\chi) = \prod_{k < j\le 12} N_j(s,\chi).
 $$

Suppose first that $1 < \alpha \le 2$. We see that~\eqref{eq:U(HKlam)llx38-lam/2} yields the desired bound
 \begin{equation}\label{eq:U(H,K,lambda)}
U(H, K, \lambda) \ll x^{1/2-\delta} R^{1/\alpha}
 \end{equation}
if 
$$
\frac \lambda 2 > -\frac 18 + \frac{3\vartheta}4 + 2\delta,
$$ 
that is,
 \begin{equation}\label{eq:lambda > 3 theta/2-14}
\lambda > \frac{3\vartheta}2 - \frac 14 + 4\delta. 
 \end{equation}
Suppose that $\vartheta \le13/28 - \varepsilon$. Then
 $$
 \frac{3\vartheta}2 - \frac 14 + 4\delta \le
 \min\left\{\frac 9{20}, \, \frac 56 (1 - \vartheta)\right\}
 -\varepsilon
 $$
and~\eqref{eq:lambda > 3 theta/2-14} is a consequence of our hypothesis~\eqref{eq:min910,56,1-theta}. This completes the proof of Theorem~\ref{thm:BV4PS} in the case
 $$
 \frac{26}{23} \le \alpha \le 2.
 $$

Now suppose that $1 < \alpha < 26/23$. We get~\eqref{eq:U(H,K,lambda)} from~\eqref{eq:U(HKlam)llx38+o(1)} provided that
 $$
 \frac 38 + \frac\lambda 4 + \left(\frac 3{4\alpha} +
 \frac 14\right) \vartheta \le \frac 12 + \frac \vartheta \alpha
 - 2\delta
 $$
\textit{and}
 $$\frac 38 - \frac \lambda 2 + \frac{11\vartheta}8
 \le \frac 12 + \frac \vartheta\alpha - 2\delta.
 $$
This gives an interval of $\lambda$ in which we obtain~\eqref{eq:U(H,K,lambda)}, namely
 $$
 \vartheta\left(\frac{11}4 - \frac 2\alpha\right) - \frac 14
 + 4\delta \le \lambda \le \frac 12 - \vartheta\left(
 1 - \frac 1\alpha\right) - 8\delta.
 $$
Recalling~\eqref{eq:lambda > 3 theta/2-14}, we see that for a constant $\vartheta_0$, all $\vartheta \le \vartheta_0$ are admissible if
 $$
 \vartheta_0 \left(\frac{11}4 - \frac 2\alpha\right) -
 \frac 14 < \min \left\{\frac 9{20},\, \frac 56\,
 (1-\vartheta_0)\right\}
 $$
and 
 $$
 \frac 12 - \vartheta_0\left(1 - \frac 1\alpha\right)
 > \frac{3\vartheta_0}2 - \frac 14.
 $$
The second of these conditions is equivalent to
 \begin{equation}\label{eq:theta0 < 3alpha/10alpha}
 \vartheta_0 < \frac{3\alpha}{10\alpha -4}\, ,
 \end{equation}
and one can verify that~\eqref{eq:theta0 < 3alpha/10alpha} implies the first condition, namely
 $$
 \vartheta_0 < \min\left\{\frac{14\alpha}{55\alpha - 40},
 \, \frac{13\alpha}{43\alpha - 24}\right\}.
 $$
This completes the proof of Theorem~\ref{thm:BV4PS} for $1 < \alpha <  26/23$.

Now suppose that $\alpha > 2$. We obtain the desired bound~\eqref{eq:U(H,K,lambda)} from~\eqref{eq:U(HKlam)llx38-lam/2+0)(1)} provided that
 $$
 \frac 38 - \frac \lambda 2 + \vartheta\left(
 \frac 1{2\alpha} + 1\right) < \frac 12 +
 \frac\vartheta \alpha - 2\delta,
 $$
that is,
 $$
 \lambda > \vartheta\left(2 - \frac 1\alpha\right)
 - \frac 14 + 4\delta.
 $$
This is a consequence of~\eqref{eq:min910,56,1-theta} if
  $$
  \vartheta\left(2 - \frac 1\alpha\right) - \frac 14
  \le \min\left\{\frac 9{20}\, , \frac 56\, (1-\vartheta)
  \right\} - 2\varepsilon.
  $$
It suffices if
 $$
 \vartheta \le \min\left\{\frac{13\alpha}{34\alpha-12}
 , \, \frac{7\alpha}{20\alpha-10}\right\}-2\varepsilon.
 $$
Theorem~\ref{thm:BV4PS} now also follows for $\alpha > 2$, and the proof  is complete.

   \section{Comments}
   \label{sec:comm}
   
As we have mentioned our method also works for $k=3$ if one uses a modification of a result of Hooley~\cite{Hool} 
given in~\cite{CKMS}. Any improvements on that result can potentially make our approach more competitive for $k=3$ as well. 
Furthermore, it is quite feasible that the method of Browning~\cite{Brow} can be used to obtain a version of 
Lemma~\ref{lem:Morm}  for general polynomials and thus enable our method of proof of Theorem~\ref{thm:k>4}
 to work for general polynomial moduli. 
 
Considering only prime moduli, we can extend the level of distribution of Theorem~\ref{thm:BV4PS} up to $x^{1/2}$. Indeed, we remark that if we assume that for some $\alpha_0$, for $\alpha <\alpha_0$ 
the number of primes $p$ of the form $p = \fl{j^\alpha} \in [R, 2R]$ is of right order of magnitude, 
that is of cardinality $R^{1+o(1)}$
we can replace  the set  $\cS_\alpha(R)$  defined by~\eqref{eq:Small SaR} by the following set 
of primes
$$
\widetilde \cS_\alpha(R) =\{ \fl{j^\alpha}~\text{prime}:~j \in \N\} \cap [R, 2R].
$$
Then for the corresponding analogues $\widetilde T(\bc,\lambda)$ of the sums $T(\bc,\lambda)$ defined in~\eqref{defT} 
we only have terms with $t=1$ giving
 $$
 \widetilde T(\bc,\lambda) \le R^{1/\alpha +1/4+o(1)} N^{3/4}  \|\bc\|^2
$$
(as in~\eqref{eq:St bound 2bis} taken with $t=1$). In turn, for the following analogue 
$$
\widetilde  U(H, K) = \sum_{p \in \widetilde  \cS_\alpha(R) }\sum_{\chi \in \cX_p^*} \left|H(1/2+it,\chi)K(1/2+it,\chi)\right| 
$$
of $ U(H, K, \lambda)$ in Lemma~\ref{lem:K(s chi) = sum n sim K}
this leads to 
$$
\widetilde  U(H, K)   \le x^{3/8} R^{1/\alpha +1/4+o(1)} \le x^{1/2-\varepsilon/5} R^{1/\alpha}
$$
provided that $R \le x^{1/2 - \varepsilon}$, which is what required for our purpose. 

By the result of Rivat and Wu~\cite{RiWu} we can  take
$$
\alpha_0 = \frac{243}{205}=1.1853\ldots.
$$
In particular for $\alpha < \alpha_0$ we obtain a version of Corollary~\ref{cor:PS Div}
with any fixed $\vartheta <1/2$. 

Using integers of the form $ \fl{j^\alpha} $ without small prime divisors, as, for example,
in~\cite{Ak,BBGY,Guo}, one could derive  other versions of   Theorem~\ref{thm:BV4PS} with moduli restricted to subsequences of Piatetski-Shapiro integers.

  \section*{Acknowledgements}
The authors are grateful to Lee Zhao for his encouragement to make a general version, with a full proof, of his bound~\eqref{eq:Zhaobound} available, see Theorem~\ref{thm:f}. 
  
  During the preparation of this work, M.M. was supported by the Austrian Science Fund (FWF), projects P-33043   and 
I.E.S. by the Australian Research Council Grant DP170100786.

\end{document}